\documentclass[11pt,reqno,letterpaper]{amsart}
\usepackage[T1]{fontenc}
\usepackage{lmodern,color}
\usepackage{amsmath,amssymb,amsthm}
\usepackage[dvipsnames]{xcolor}
\usepackage{microtype}
\usepackage{enumitem}
\usepackage[colorlinks=true,allcolors=blue]{hyperref}
\usepackage[left=1.25in,right=1.25in,top=1.1in,bottom=1.1in]{geometry}
\usepackage{todonotes}

\theoremstyle{plain}
\newtheorem{theorem}{Theorem}[section]
\newtheorem{lemma}[theorem]{Lemma}
\newtheorem{conjecture}[theorem]{Conjecture}
\newtheorem{proposition}[theorem]{Proposition}
\newtheorem{corollary}[theorem]{Corollary}

\theoremstyle{definition}
\newtheorem{definition}[theorem]{Definition}
\newtheorem{remark}[theorem]{Remark}

\newcommand{\RR}{\mathbb{R}}
\renewcommand{\SS}{\mathbb S}
\newcommand{\calK}{\mathcal{K}}
\newcommand{\V}{\operatorname{V}}
\renewcommand{\S}{\operatorname{S}}
\newcommand{\D}{\operatorname{D}}
\newcommand{\supp}{\operatorname{supp}}
\newcommand{\relint}{\operatorname{relint}}
\newcommand{\relbd}{\operatorname{relbd}}

\newcommand{\aff}{\operatorname{aff}}
\newcommand{\codim}{\operatorname{codim}}
\newcommand{\conv}{\operatorname{conv}}

\newcommand{\proj}{\operatorname{P}}

\DeclareMathOperator{\interior}{int}
\DeclareMathOperator{\Sym}{Sym}

\newenvironment{enuma}{
	\begin{enumerate}[label=\textnormal{(\alph*)},leftmargin=*,align=left,itemindent=18pt,labelwidth=0pt,labelsep=0pt, listparindent=\parindent,parsep=\parskip,topsep=2pt,itemsep=2pt,partopsep=4pt]}
{	
	\end{enumerate}
}

 \title[Higher-order Godbersen inequality]{Higher-Order Godbersen Inequality and a characterization of simplices via the face lattice}

\author[F. Fry\v{s}]{Filip Fry\v{s}}
\address{Mathematical Institute, Faculty of Mathematics and Physics, Charles University, Prague, Czechia}
\email{filip.frys@matfyz.cuni.cz}

\author[J. Kotrbat\'y]{Jan Kotrbat\'y}
\address{Mathematical Institute, Faculty of Mathematics and Physics, Charles University, Prague, Czechia}
\email{kotrbaty@karlin.mff.cuni.cz}

\author[M. A. Mouamine]{Mohamed A. Mouamine}
\address{Mathematical Institute, Faculty of Mathematics and Physics, Charles University, Prague, Czechia}
\email{mohamed-abdeldjalil.mouamine@matfyz.cuni.cz}

\author[S. Wang]{Shouda Wang}
\address{Department of Mathematics, KTH Royal Institute of Technology}
\email{shoudawang@princeton.edu}

\begin{document}

\begin{abstract}
We settle the long-standing Godbersen conjecture that the mixed volume of a convex body of fixed positive volume with its reflection about the origin is maximized precisely by simplices. For the proof, we establish a new characterization of simplices by a natural incidence relation on the face lattice. Moreover, we establish a higher-order generalization of the Godbersen inequality and show that simplices are the only nontrivial equality cases,
confirming a conjecture of Schneider from 2000.
\end{abstract}

\maketitle

\section{Introduction}

Simplices are among the most fundamental objects in convex geometry and their characterization has been a central theme of study in the field.
One important approach of characterization is to define  the simplex as the unique extremizer of a geometric inequality. A first example is the Minkowski asymmetry measure of a full-dimensional convex body $K\subseteq\RR^n$ given by
$$
s(K):=\inf\{t>0 : -K\subseteq tK+x\text{ for some }x\in\RR^n\},
$$
where $-K$ is the reflection of $K$ about the origin. It satisfies $s(K)\leq n$ and equality holds if and only if $K$ is a simplex, see \cite[§6.1]{Grunbaum63}.
Further inequalities having simplices as their minimizers include Ball's reverse isoperimetric inequality \cite{Ball91},  Grünbaum's inequality \cite{Grunbaum63},   Zhang's projection inequality \cite{Zhang91}, the  Mahler inequality \cite{Mahler39}, and many more.

Our present work is motivated by yet another such inequality. For a convex body $K\subseteq\RR^n$ we write $\V_n(K)$ for its volume and define the difference body $K-K:=\{x-y:x,y\in K\}$. The celebrated Rogers--Shephard inequality \cite{RS57,RS58} bounds the volume of the difference body by the volume of $K$ itself:
\begin{equation}\label{eq:rogers-shephard}
 \V_n(K-K)\leq\binom{2n}{n}\V_n(K).
\end{equation}
If $K$ is full-dimensional, then equality holds in \eqref{eq:rogers-shephard} if and only if $K$ is a simplex. Let us also mention that  the Brunn--Minkowski inequality  gives a sharp lower bound on $\V_n(K-K)$, namely $\V_n(K-K)\geq2^n\V_n(K)$, with equality if and only if $K$ is centrally symmetric or lower-dimensional.

A conjecture of Godbersen \cite{G38}, which predates the proof of the Rogers--Shephard inequality, speculated  a refinement of \eqref{eq:rogers-shephard}. Let $\V_n$ denote the mixed volume, i.e., the Minkowski-multilinear, symmetric polarization of the volume. Since $K-K=K+(-K)$,  the multilinearity of mixed volume yields $\V_n(K-K)=\sum_{k=0}^n\binom{n}{k}\V_n(K[k],\allowbreak-K[n-k])$, where $K[k]$ denotes $k$ repetitions of $K$. On the other hand,  $\binom{2n}{n}=\sum_{k=0}^n\binom{n}{k}^2$. Comparing the two sums, Godbersen \cite{G38} in 1938, and independently Makai Jr. \cite{HM74}  in 1974, conjectured that \eqref{eq:rogers-shephard} in fact holds term-wise; see \eqref{eq:godbersen} below.

\emph{A priori}, this stronger family of inequalities may have many more extremizers than \eqref{eq:rogers-shephard}. In addition, it is often the case that inequalities involving mixed volumes have a large number of intricate equality cases that are not straightforward to describe, cf. \cite{SvH19,SvH20,vHW25,CP24,SZ16,MS24,vHYZ24}.
Nevertheless, it was commonly conjectured that the only nontrivial equality cases should still be simplices, see, e.g., \cite[§23.1.3]{BuragoZalgaller} and \cite[§10.1]{Sch14}. The following inequality together with equality characterization is thus commonly referred to as the \emph{Godbersen conjecture}.

\begin{conjecture}[Godbersen conjecture]
\label{con:G}
Let $0\leq k\leq n$. For any convex body $K\subseteq\RR^n$, one has
\begin{equation}\label{eq:godbersen}
 \V_n\bigl(K[k],-K[n-k]\bigr)\leq\binom{n}{k}\V_n(K).
\end{equation}
If $1\leq k\leq n-1$ and $K$ is full-dimensional, then equality holds if and only if $K$ is a simplex.
\end{conjecture}

Note that if $K$ is lower-dimensional, then both sides of \eqref{eq:godbersen} are equal to zero. If $k\in\{0,n\}$ then the mixed volume on the left-hand side of \eqref{eq:godbersen} is just the volume of $K$ and hence equality holds for any $K$. Moreover, the cases $k\in\{1,n-1\}$ follow at once from the inequality $s(K)\leq n$  for the Minkowski asymmetry measure.

Despite its long history, most of the progress on the Godbersen conjecture has only been made rather recently. This includes a weaker general bound on the left-hand side of \eqref{eq:godbersen} due to Artstein-Avidan, Einhorn, Florentin, and Ostrover \cite{AAEFO15}, a proof for anti-blocking and locally anti-blocking convex bodies by Artstein-Avidan, Sadovsky, and Sanyal \cite{AASS23,S25}, and  averaged forms of the conjecture due to Artstein-Avidan and Putterman \cite{AAP25}. Very recently, the second- and third-named authors \cite{KM26} proved the inequality \eqref{eq:godbersen} for all convex bodies and characterized equality among polytopes.

In the present paper, we completely characterize the equality cases of \eqref{eq:godbersen} for general convex bodies and thus settle the Godbersen conjecture in full generality. Furthermore, we prove a more general version of Conjecture \ref{con:G}, where $K$ and $-K$ are replaced by an array of natural embeddings of $K$ into a product space, which was conjectured by Schneider \cite{Sch00} in 2000. 
A particularly interesting by-product of the proof is  a new characterization of the simplex via the structure of its face lattice.

\subsection{Equality cases of the Godbersen inequality}

Our first main result resolves the equality cases part of Conjecture \ref{con:G}.

\begin{theorem}\label{thm1.1}
Let $K\subseteq\RR^n$ be a full-dimensional convex body and let $1\leq k\leq n-1$.
Equality in \eqref{eq:godbersen} holds if and only if $K$ is a simplex.
\end{theorem}

In order to outline the strategy of our proof of Theorem \ref{thm1.1}, let us briefly recall the proof of \eqref{eq:godbersen} given in \cite{KM26}.
The proof repeatedly uses the following monotonicity property for mixed volumes: for convex bodies $L,M,C_1,\dots, C_{n-1}\subseteq\RR^n$ with $L\subseteq M$, one has
\begin{equation}
\label{eq:monotonicity}
\V_n(L,C_1,\dots, C_{n-1})\leq \V_n(M,C_1,\dots, C_{n-1}).
\end{equation}
Write $\iota_1,\iota_2:\RR^n\rightarrow (\RR^n)^2$ for the canonical embeddings into the first and second factors, respectively.
The body $K^2=K\times K\subseteq (\RR^n)^2$ can then be written as $K^2=\iota_1K+\iota_2K$.
Denote by $\Delta: \RR^n \rightarrow (\RR^n)^2$ the diagonal embedding given by $\Delta(x)=(x,x)$, $x\in\RR^n$. 
Then for any convex body $K\subseteq\RR^n$ we have $\Delta K\subseteq K^2$ and hence
\begin{align}
 \label{eq4} 
\begin{split}
    \V_n(K) \V_n  \bigl( K[k],-K[n-k]\bigr)&=\binom{2n}{n} \V_{2n}\bigl(\Delta K[n],\iota_1K[k],\iota_2K[n-k]\bigr) \\
   & \leq  \binom{2n}{n} \V_{2n}\bigl(K^2[n],\iota_1K[k],\iota_2K[n-k]\bigr)\\
&= \binom{n}{k}\V_n(K)^2,
   \end{split}
\end{align}
where the inequality uses \eqref{eq:monotonicity} $n$ times and the equalities follow from the standard projection formula for mixed volumes (see \cite{KM26} and Section \ref{sec3} below for more details). If $K$ is full-dimensional, dividing by $\V_n(K)>0$ proves \eqref{eq:godbersen}. In addition, equality in \eqref{eq:godbersen} holds if and only if the monotonicity inequalities used in \eqref{eq4} all hold with equality.

It is well known that equality in the monotonicity inequality \eqref{eq:monotonicity} can be reformulated as a condition in terms of the support of mixed area measures, see Lemma \ref{lem2.11}. In our case, this implies that $K$ is an extremizer of \eqref{eq4} if and only if the support functions of $\Delta K$ and $K^2$ are equal on the support of certain mixed area measures, more precisely,
\begin{equation}
h_{\Delta K} = h_{K^2} \quad \text{on}\quad \supp \S_{B[n-1], \iota_1 K[k], \iota_2 K[n-k]},
\label{eq5}
\end{equation}
see Section \ref{sec2} for notation and more details. 

While an explicit characterization of the support of the mixed area measure is a wide open problem in general, it turns out that in this special case, equality in the monotonicity inequality admits a simple geometric interpretation. More precisely, 
we prove a local version of the aforementioned projection formula, with mixed volumes replaced by mixed area measures (see Proposition \ref{prop:local_proj}), which allows us to reformulate \eqref{eq5} as a condition on the touching cones of $K$ (see Section \ref{sss:normal} for definition and more details). These constitute the first step of the proof of Theorem \ref{thm1.1} and result in the following proposition.

\begin{proposition}\label{prop1.3}
Let $1\leq k\leq n-1$. Assume $K\subseteq\RR^n$ is a full-dimensional convex body.
Then $K$  attains equality in \eqref{eq:godbersen} if and only if for any two directions $u,v\in \RR^n$ with $\dim T(K,u)\leq k$ and $\dim T(K,v)\leq n-k$, there exists a direction $w\in \RR^n$ such that $T(K,w)$ contains both $T(K,u)$ and $T(K,v)$. 
\end{proposition}

    Now assume that $K$ contains the origin in its interior. The fact that touching cones and faces are dual notions under  the operation of polarity (see Section \ref{sec2}) then allows us to translate the conclusion in Proposition \ref{prop1.3} to a condition on the face lattice of the polar body $K^\circ$. This states that for any two faces $F$ and $G$ of $K^\circ$ with $\dim F\leq k-1$ and $\dim G\leq n-k-1$, there exists a proper face of $K^\circ$ containing both of them. The second step in our classification of equality cases in \eqref{eq:godbersen} is that this face lattice condition already characterizes simplices.

\begin{theorem}\label{thm:equivalent}
Let $K\subseteq\RR^n$ be a full-dimensional convex body and fix an integer $1\leq k\leq n-1$.
The following are equivalent:
\begin{enuma}
\item  $K$ is a simplex.
\item For any two faces $F$ and $G$ of $K$ with  $\dim F\leq k-1$ and $\dim G\leq n-k-1$, there exists a proper face $H$ of $K$ containing both $F$ and $G$.
\item  For any two directions $u,v\in \RR^n$ with $\dim T(K,u)\leq k$ and $\dim T(K,v)\leq n-k$, there exists a direction $w\in \RR^n$ such that $T(K,w)$ contains both $T(K,u)$ and $T(K,v)$. 
\end{enuma}
\end{theorem}

We remark that Theorem \ref{thm:equivalent}  establishes a characterization of simplices entirely in terms of their facial or normal structure.  The only  previous result in this direction known to us is due to  Develin \cite[Theorem 1]{Develin04}, who proved  Theorem \ref{thm:equivalent} when $K$ is a polytope. The core difficulty of Theorem \ref{thm:equivalent}, however, is that the face lattice of a general convex body is possibly  infinite. Yet, our proof strategy is motivated by Develin's argument; more precisely, we inductively reduce the statement to the case $k=1$ where a simpler direct proof is available.

Theorem \ref{thm1.1} now follows at once by combining Theorem \ref{thm:equivalent} with Proposition \ref{prop1.3}.

\subsection{Higher-order Godbersen's inequality and equality cases}

The key insight of the proof of \eqref{eq:godbersen} given in \cite{KM26} and outlined above is the embedding of the problem to the product space. This idea originated in the previous study of a certain natural generalization of Conjecture \ref{con:G} by the first- and second-named authors \cite{K25,FK26} as we now explain.

For $p\geq1$, let $\Delta_p:\RR^n\to(\RR^n)^p$ be the diagonal map $\Delta_p(x)=(x,\ldots,x)$ and let $\iota_i:\RR^n\to(\RR^n)^p$ be the embedding of $\RR^n$ into the $i$-th copy of $\RR^n$ in $(\RR^n)^p$. Then the $p$-th Cartesian product of $K\subseteq\RR^n$ equals $K^p=\iota_1K+\cdots+\iota_pK$.
We denote by $\V_{pn}$ the mixed volume in $\RR^{pn}$.
Schneider \cite[Satz 2]{S70} proved the following higher-order version of  the Rogers--Shephard inequality:
\begin{equation}\label{eq:schneider}
    \V_{pn}\bigl(K^p-\Delta_pK\bigr)\leq \binom{pn+n}{n}\V_n(K)^p,
\end{equation} 
 where equality holds  if and only if $K$ is lower-dimensional or a simplex.

Using the multilinearity of mixed volume on the left-hand side of \eqref{eq:schneider} and a suitable Vandermonde expansion of ${pn+n\choose n}$ on the right, one may again expect that the inequality \eqref{eq:schneider} holds term-wise, with simplices being the only nontrivial extremizers. This higher-order Godbersen inequality was first conjectured by Schneider \cite[p.~537]{Sch00} and later independently by the second-named author \cite[Conjecture 1.3]{K25}. Previous results on this conjecture were essentially limited to the generalization of the known cases of the original Godbersen conjecture to the higher-order setting, with partial contributions by Janson (see \cite{Sch00}), the second-named author \cite{K25}, and the first- and second-named authors \cite{FK26}.

Our second main result resolves the higher-order Godbersen conjecture in full generality.

\begin{theorem}\label{thm:higherGodbersen}
Let $p\geq1$ and $0\leq k_0,\dots, k_p\leq n$ be integers such that $\sum_{i=0}^pk_i=n$. For any convex body $K\subseteq\RR^n$ we have
\begin{equation}\label{eq:higher-godbersen}
 \V_{pn}\bigl(-\Delta_p K[n-k_0],\iota_1K[n-k_1],\ldots,
                    \iota_pK[n-k_p]\bigr)
 \leq\frac{(n!)^p}{(pn)!}\binom{n}{k_0,\ldots,k_p}\V_n(K)^p.
\end{equation}
If $k_i<n$ for all $i$ and $K$ is full-dimensional, then equality holds if and only if $K$ is a simplex.
\end{theorem}

\begin{remark}\label{rem:reduction}
If $k_i=n$ for some $i$, then both sides of \eqref{eq:higher-godbersen} are equal to $\frac{(n!)^p}{(pn)!}\V_n(K)^p$,
see \cite[Remark 3.5 (b)]{K25},
hence equality is attained for any $K$. Similarly, if $K$ is lower-dimensional, then both sides of \eqref{eq:higher-godbersen} are always equal to zero.

Furthermore, by \cite[Lemma 3.1]{FK26}, zero indices can be eliminated without changing whether equality holds. Thus, apart from the trivial cases above, it suffices to consider $k_0, \dots, k_p\geq1$, which in particular implies $p\leq n-1$.
\end{remark}

Note that for $p=1$, Theorem \ref{thm:higherGodbersen} subsumes the Godbersen inequality \eqref{eq:godbersen} and its equality characterization Theorem \ref{thm1.1}. It was pointed out in \cite{KM26} that the proof of \eqref{eq:higher-godbersen} follows from the same monotonicity argument as in the case $p=1$. We give the proof of the general inequality \eqref{eq:higher-godbersen} in Section \ref{sec3}.

The proof of the equality cases of \eqref{eq:higher-godbersen} follows a similar strategy as Theorem \ref{thm1.1}. First, we show the following analogue of Proposition \ref{prop1.3}.

\begin{proposition}\label{prop1.6}
Let $1\leq p\leq n-1$ and fix integers $k_0,\dots, k_p\geq 1$ with $\sum_{i=0}^pk_i=n$. Assume $K\subseteq\RR^n$ is a full-dimensional convex body. Then $K$ attains equality in \eqref{eq:higher-godbersen} if and only if  for any $p+1$ directions $u_0,\dots, u_{p}\in \RR^n$ with $\dim T(K,u_i)\leq k_i$, there exists a direction $w\in \RR^n$ such that $T(K,u_i)\subseteq T(K,w)$ for all $i=0,\dots,p$. 
\end{proposition}

We then inductively extend  Theorem \ref{thm:equivalent} from two faces (touching cones) to $p+1$ faces (touching cones) in the following way.

\begin{theorem}\label{thm:higher_equivalent}
Let $K\subseteq\RR^n$ be a full-dimensional convex body. Let $1\leq p\leq n-1$ and fix integers $k_0,\ldots,k_p\geq 1$ with $\sum_{i=0}^p k_i=n$.
Then the following are equivalent:
\begin{enuma}
\item $K$ is a simplex.
\item For any faces $F_0,\ldots,F_p$ of $K$ with $\dim F_i\leq k_i-1$, $i=0,\dots,p$, there exists a proper face $F$ of $K$ such that $F_i\subseteq F$ for all $i=0,\dots,p$.
\item For any $p+1$ directions $u_0,\dots, u_{p}\in \RR^n$ with $\dim T(K,u_i)\leq k_i$, $i=0,\dots,p$, there exists a direction $w\in \RR^n$ such that $T(K,u_i)\subseteq T(K,w)$ for all $i=0,\dots,p$.
\end{enuma}
\end{theorem}
Combining Proposition \ref{prop1.6} with Theorem \ref{thm:higher_equivalent} then completes the equality characterization of the higher-order Godbersen inequality \eqref{eq:higher-godbersen}, and thus proves Theorem \ref{thm:higherGodbersen}.

\subsection{Relation to other recent work}
Brandenberg, González Merino, and Grundbacher \cite{BGMG}, using an argument from \cite{KM26} together with a consequence of \eqref{eq:higher-godbersen}, previously conjectured in \cite[Conjecture 1.2]{K25}, recently proved an $L_q$-version of Schneider's inequality \eqref{eq:schneider} for general $q\in[1,\infty]$. Moreover, they also established an optimal stability result for their inequality. This in particular subsumes a stability version of \eqref{eq:schneider}, generalizing to the higher-order setting the stability of the Rogers--Shephard inequality previously proven by Böröczky \cite{boroczky05}.

By the time the present paper was completed, we learned that Friedland, Putterman, and Szusterman \cite{FPS} also obtained the equality characterization for the Godbersen inequality \eqref{eq:godbersen} using an approach significantly different from ours. Notably, their approach is based on using algebraic manipulations to reduce the problem of characterizing equality cases of \eqref{eq:godbersen} to the case $k=1$, where Conjecture \ref{con:G} was known to hold. Although we use a similar reduction scheme to prove Theorem \ref{thm:equivalent}, our method is more geometric in nature and deals with the equality cases for each $k=1,\dots,n-1$. 
In a subsequent update, Friedland, Putterman, and Szusterman \cite{FPS} extended their approach to the higher-order Godbersen inequality \eqref{eq:higher-godbersen}, providing another proof of our Theorem \ref{thm:higherGodbersen}.

\medskip\subsection*{Organization of the paper.}
Section \ref{sec2} recalls facial structure, normal structure,   mixed volumes, and mixed area measures.
Section \ref{sec3} proves the inequality part in Theorem \ref{thm:higherGodbersen}.
Section \ref{sec4} proves the local projection formula for mixed area measures,   and Propositions  \ref{prop1.3} and \ref{prop1.6}.
Section \ref{sec5} proves Theorem \ref{thm:equivalent} and thus together with Section \ref{sec4} completes the proof of Theorem \ref{thm1.1}.
Section \ref{sec6} proves  Theorem \ref{thm:higher_equivalent} and combined with Section \ref{sec4}, proves the equality characterization in Theorem \ref{thm:higherGodbersen}.

\subsection*{Acknowledgments}
We thank Károly Böröczky, Eli Putterman, Maud Szusterman, and Ramon van Handel for useful discussions.

FF was supported by Charles University grants PRIMUS/24/SCI/009 and GAUK/34126. JK was supported by Charles University grants PRIMUS/24/SCI/009 and UNCE/24/SCI/022. MM was supported by Charles University grant PRIMUS/24/SCI/009. SW was partially supported by the Knut and Alice Wallenberg Foundation.

\subsection*{AI disclosure}
The authors used ChatGPT to assist with secondary tasks such as checking the proofs, correcting typos, and finding references.
The mathematical contents are due to the authors, who take full responsibility for the paper.

\section{Preliminaries}
\label{sec2}

A convex body is a nonempty compact convex set, possibly of lower dimension. We denote by $\calK^n$ the set of all convex bodies in $\RR^n$.
Since the polar body is an important tool in this paper,  
we will often work with the class $\calK^n_o$ of full-dimensional convex bodies in $\RR^n$ containing the origin in their interior. Throughout the paper we assume $n\geq2$.

\subsection{Facial and normal structure}

\label{ss:facial_normal}

We  first recall basic notions about the facial structure and normal structure of convex bodies and
refer to \cite[§2.1--2.2]{Sch14}, 
\cite[§18]{Roc70}, \cite{W12}, and \cite{vHW25} for a more comprehensive treatment.

\subsubsection{Facial structure}
\label{sss:facial}

For two points $x,y\in \RR^n$, denote by $[x,y]$ the closed line segment from $x$ to $y$ and denote by $(x,y)$ the open line segment from $x$ to $y$.
A \emph{face} of a convex body $K\in\calK^n$ is a nonempty convex subset $F\subseteq K$ such that, whenever $x,y\in K$ and $(x,y)\cap F\neq \emptyset$, the whole segment $[x,y]$ lies in $F$.
In particular, $K$ is a face of itself, and any face other than $K$ is called a \emph{proper face}. 
Note that a face of a face of $K$ is again a face of $K$.
Also, if $F$ and $G$ are two faces of $K$, then $F\cap G$, if nonempty, is a face of $F$, $G$, and $K$.

For any nonempty subset $A\subseteq K$, we define $F(K,A)$ as the smallest (in the order of inclusion) face of $K$ containing $A$.
Alternatively, $F(K,A) = \bigcap_{F\supseteq A}  F$, where the intersection runs over all faces $F$ of $K$ containing $A$.
When $A = \{x\}$ for some $x\in K$, we write simply $F(K,x)$ instead of $F(K,\{x\})$.\footnote{
In previous literature, the notation $F(K,u)$ is often used for the \emph{exposed} face in the direction of $u$. 
However, we reserve it for general faces since exposed faces are almost never used in the present paper. }

\begin{lemma}\label{lem:minimal-face}
The face $F(K,x)$ is the unique face of $K$ containing $x$ in its relative interior.
\end{lemma}
\begin{proof}
    First, we have  $x\in \relint F(K,x)$. Indeed, if $x\in \relbd F(K,x)$, then we may choose a proper face $G$ of  $F(K,x)$ containing $x$. Then $G$ is a face of $K$ which is strictly contained in $F(K,x)$, contradicting the definition of $F(K,x)$.
    Now we show the uniqueness. 
    Let $F'$ be any face  with $x\in \relint F'$. Then $x\in F(K,x)\cap F'$. The definition of $F(K,x)$ implies $F(K,x)\subseteq F'$. Therefore, $F(K,x)$ is a face of $F'$. But if $F(K,x)$ is a proper face of $F'$, then  $x\in F(K,x)\subseteq \relbd F'$, contradicting the choice of $F'$. Hence $F' = F(K,x)$ as desired.
\end{proof}

A point $x\in K$ is an \emph{extreme point} if $\{x\}$ is a face of $K$. More generally, a point  $x\in K$ is called \emph{$p$-extreme} if $\dim F(K,x)\leq p$. In particular, 0-extreme points are exactly extreme points.
We recall that Minkowski's theorem asserts that a convex body is the convex hull of its extreme points \cite[Corollary 1.4.5]{Sch14}.

Consider the collection of all the  faces of $K$ together with the empty set, ordered by inclusion. This forms a lattice, which we call the \emph{face lattice} of $K$. For two faces $F,\,G$ of $K$, their \emph{join} $F\vee G$ is therefore the smallest face containing both of them, which may be $K$. In other words, $F\vee G =F(K, F\cup G)$.
Their \emph{meet} is $F\wedge G=F\cap G$, which may be empty.
When $F = \{x\}$ is a singleton, we  write simply $F\vee G = x\vee G$.

The \emph{support function} of a convex body $K\in\calK^n$ is defined by $h_K(y):=\max_{x\in K}\langle y,x\rangle$, $y\in\RR^n$. Observe that $h_K$ is 1-homogeneous. The \emph{exposed face} of $K$ in the direction of $u\in\SS^{n-1}$ is given by $\{x\in K:\langle u,x\rangle=h_K(u)\}$. An exposed face is always a face but not vice versa.
While this will not be used in the paper, we note that exposed faces together with the empty set also form a lattice.

We record two elementary facts about faces that will be repeatedly used throughout the proof of Theorem \ref{thm:equivalent} in Section \ref{sec5}.

\begin{lemma}\label{lemma2.1}
Let $F$ and $G$ be faces of $K\in\calK^n$.
If $F\cap \relint G\neq  \emptyset$, then $G\subseteq F$.
\end{lemma}
\begin{proof}
Fix $x\in F\cap \relint G$.
For any $a\in G\setminus\{x\}$, there exists $b\in G$ such that $x\in(a,b)$.
Since $F$ is a face, we have $a,b\in F$.
Since $a$ was chosen arbitrarily in $G$, it follows that $G\subseteq F$.
\end{proof}

\begin{lemma}\label{lemma2.2}
Let $F$ and $G$ be faces of $K\in\calK^n$, and let $u\in \relint F$ and $v\in \relint G$.
Then $(u,v)\subseteq \relint (F\vee G)$.
\end{lemma}
\begin{proof}
Let $x\in (u,v)$ and consider $H:=F(K,x)$. Lemma \ref{lem:minimal-face} implies $x\in \relint H$.
Since $x\in (u,v)$, we have $u,v\in H$.
Hence by Lemma \ref{lemma2.1} we have $F,G \subseteq H$.
Since $x\in F\vee G$, we have $H = F\vee G$.
Therefore $x\in \relint (F\vee G)$ and the claim follows.
\end{proof}

\subsubsection{Normal structure and duality}
\label{sss:normal}
This entire section is dual to Section \ref{sss:facial} under taking the polar body, see below. We first recall normal cones, which are dual to exposed faces. Fix a  convex body $K\in\calK^n$.
The \emph{normal cone} at $x\in  K$ is
$$
N(K,x):=\{v\in\RR^n:\langle v,y-x\rangle\leq0 \text{ for every }y\in K\}.
$$
Second, we recall touching cones, which are dual to faces.
For $u\in\RR^n$,
fix any $x\in K$ with $\langle x,u\rangle  = h_K(u)$ and define the \emph{touching cone} $T(K,u)$ as the smallest face of $N(K,x)$ containing $u$.
It can be seen that this does not depend on the choice of $x$.
In this paper, we   consider $\{0\}$ as a touching cone or a normal cone: when $K$ is full-dimensional,  $N(K,x) = \{0\}$ for $x\in \interior K$ and $T(K,0)=\{0\}$.
We call a nonzero touching (normal) cone a \emph{proper} touching (normal) cone.
Lastly, we recall $p$-extreme directions, which are dual to $p$-extreme points.
A direction $u\in \SS^{n-1}$ is called \emph{$p$-extreme} if $\dim T(K,u)\leq p+1$.

While normal and touching cones are defined for all convex bodies, in order to introduce the duality we further assume $K\in\calK^n_o$. The polar (dual) body $K^\circ:=\{v\in\RR^n:h_K(v)\leq1\}$ also belongs to $\calK^n_o$.
The following fact, due to Weis \cite{W12}, establishes the duality between touching cones  (normal cones resp.) of $K$ and faces (exposed faces resp.) of $K^\circ$.

\begin{lemma}[{\cite[Proposition 4.7 and Theorem 8.3]{W12}}]
\label{lem:fact}
Let $K\in\calK^n_o$. There exist bijections
$$
\begin{aligned}
\{\text{proper faces of }K^\circ\}&\overset{\sim}{\longrightarrow}\{\text{proper touching cones of }K\},\\
\{\text{proper exposed faces of }K^\circ\}&\overset{\sim}{\longrightarrow}\{\text{proper normal cones of }K\},
\end{aligned}
$$
given by $F\mapsto\RR_{\geq0}F$. Moreover, the map $T\mapsto T\cap\partial K^\circ$ gives the inverse to both maps above.
\end{lemma}

\begin{remark}
Using the first bijection of Lemma \ref{lem:fact}, we may immediately transfer all the properties of faces to touching cones. 
For example, Lemma \ref{lem:minimal-face} implies that $u\in \relint T(K,u)$. In addition,  for any $v\in T(K,u)\setminus\{0\}$,  $T(K,v)$ is the face of $T(K,u)$ containing $v$ in the relative interior, cf. \cite[Lemma 2.2]{vHW25}. 
Note also that $u\in \SS^{n-1}$ is a $p$-extreme direction for $K\in\mathcal{K}^n_o$ if and only if $\RR_{\geq 0}u\cap \partial K^\circ$ is a $p$-extreme point of $K^\circ$.
\end{remark}
Finally, we recall the relation between the touching cones of a projection with the touching cones of the original body.
Write $T_E$ for the cones computed in a subspace $E\subseteq\RR^n$.
If $u\in E\cap \SS^{n-1}$, then 
\begin{align}\label{eq7}
    T_E(\proj_E K,u)=T(K,u)\cap E,
\end{align}
where $\proj_E:\RR^n\to E$ is the orthogonal projection. This follows by applying \cite[Lemma 2.4]{vHW25} iteratively.

\subsection{Mixed volumes and mixed area measures}
\label{ss:mixed}

Now we define mixed volumes and mixed area measures, and recall their properties which will be needed in later sections. The standard reference for this material is \cite[Chapter 5]{Sch14}.

For convex bodies $K_1,\dots,K_m\in\calK^n$ and $t_1,\dots,t_m\geq0$, the volume
$$
\V_n\left(\sum_{i=1}^m t_iK_i\right)=\sum_{i_1,\dots,i_n=1}^m t_{i_1}\cdots t_{i_n}\V_n(K_{i_1},\dots,K_{i_n})
$$
is a homogeneous polynomial whose symmetric coefficients $\V_n(K_{i_1},\dots,K_{i_n})$ are called \emph{mixed volumes}.
Mixed volumes are nonnegative, translation invariant, and multilinear.
Volume can be recovered from mixed volumes by $\V_n(K)=\V_n(K[n])$, where $K[r]$ denotes $r$ repeated copies of $K$.
More generally, when the convex bodies are contained in a subspace $E\subseteq\RR^n$, we write $\V_{\dim E}^E$ for volume and mixed volume computed within $E$.
In particular, $\V_n = \V_n^{\RR^n}$.

It is easily seen from the definition that mixed volumes are affine covariant: for any affine map $\phi: \RR^n\rightarrow \RR^n$, where $\phi(x)=Ax+b$, and bodies $K_1,\dots,K_n\in\calK^n$, we have
\begin{equation}
\label{eq:mv_covariance}
\V_n(\phi K_1,\dots,\phi K_n)=|\det A|\V_n(K_1,\dots,K_n).
\end{equation}
We also recall the projection formula for mixed volumes \cite[Theorem~5.3.1]{Sch14}.
Let $E\subseteq\RR^n$ be a $k$-dimensional subspace, where $1\leq k\leq n-1$.
Then for any $L_1,\dots,L_k,M_1,\dots,M_{n-k}\in\calK^n$ with $L_i\subseteq E$ for $i=1,\dots,k$, one has
\begin{equation}\label{eq:MVprojection}
\binom{n}{k}\V_n(L_1,\dots,L_k,M_1,\dots,M_{n-k})=\V_k^E(L_1,\dots,L_k)\V_{n-k}^{E^\perp}(\proj_{E^\perp}M_1,\dots,\proj_{E^\perp}M_{n-k}).
\end{equation}
We will also need a similar formula for multiple linearly independent subspaces.

\begin{lemma}\label{lem:exact}
Let $E_1,\dots,E_m\subseteq\RR^n$ be linearly independent subspaces with $d_i:=\dim E_i\geq1$ and $\sum_{i=1}^m d_i=n$.
For convex bodies $K_i\subseteq E_i$, one has
\begin{align}
\label{eq:exact}
\binom{n}{d_1,\dots,d_m}\V_n(K_1[d_1],\dots,K_m[d_m])=|\det T|\prod_{i=1}^m\V_{d_i}^{E_i}(K_i),
\end{align}
where $T:\RR^n\to\bigoplus_{i=1}^m E_i$ is given by $Tx=(\proj_{E_i}x)_{i=1}^m$.
\end{lemma}
\begin{proof}
By \eqref{eq:mv_covariance}, for $t_1,\dots,t_m\geq0$ we have
$$
\V_n\left(\sum_{i=1}^m t_iK_i\right)=|\det T|\prod_{i=1}^m t_i^{d_i}\V_{d_i}^{E_i}(K_i).
$$
Comparing the coefficients of $t_1^{d_1}\cdots t_m^{d_m}$ proves the claim.
\end{proof}

For convex bodies $C_1,\dots, C_{n-1}\in\calK^n$, their \emph{mixed area measure} $\S_{C_1,\dots, C_{n-1}}$ is defined as the unique measure on $\mathbb S^{n-1}$ such that 
\begin{equation}\label{eq:area-integral}
\V_n(K,C_1,\dots,C_{n-1}) = \frac{1}{n}
\int_{\SS^{n-1}} h_K d\S_{C_1,\dots, C_{n-1}}
\end{equation}
holds for all $K\in\calK^n$.
Similar to mixed volumes,
mixed area measures are nonnegative, translation invariant, and multilinear.
When the bodies $C_1,\dots,C_{d-1}$ are contained in a $d$-dimensional subspace $E\subseteq \RR^n$, we will denote the mixed area measure by $\S^E_{C_1,\ldots,C_{d-1}}$  to specify the ambient space.

By \eqref{eq:mv_covariance}, mixed area measures satisfy the following covariance formula under invertible linear maps, cf. \cite[Equation (5.24)]{Sch14}.
For any invertible linear map $A:\RR^n\to\RR^n$, convex bodies $C_1,\dots,C_{n-1}\in\calK^n$, and any continuous positively $1$-homogeneous function $f:\RR^n\to\RR$, we have
\begin{equation}\label{eq:area-covariance}
\int_{\mathbb S^{n-1}} f(v)\,d\S_{AC_1,\dots,AC_{n-1}}(v)
=
|\det A|\int_{\mathbb S^{n-1}} f(A^{-t}u)\,d\S_{C_1,\dots,C_{n-1}}(u),
\end{equation}
where $A^{-t}$ denotes the inverse transpose of $A$.

\subsubsection{Mixed discriminants}
For convex bodies with $C^2$ support functions, the mixed area measure admits an explicit representation in terms of mixed discriminants.
Denote by $\Sym(n)$ the set of symmetric $n$-by-$n$ matrices.
The \emph{mixed discriminant} $\D_n$ is defined as the unique symmetric multilinear function  on $\Sym(n)^n$ such that $\D_n(M,\dots, M) = \det (M)$ 
holds for all $M\in \operatorname{Sym(n)}$.
Now, for $C_i\in\calK^n$ such that $h_{C_i}\in C^2(\SS^{n-1})$, their mixed area measure has a density with respect to the Hausdorff measure on the sphere $\mathbb S^{n-1}$; see \cite[Equation (2.68)]{Sch14}. More precisely,  for $u\in \mathbb S^{n-1}$, one has
\begin{equation}
\label{eq:MDrep}
    \frac{d\S_{C_1,\dots, C_{n-1}}}{d \mathcal{H}^{n-1}} (u) = \D_{n-1}(D^2h_{C_1}(u),\dots, D^2h_{C_{n-1}}(u)),
\end{equation}
where $D^2 h_C\in\Sym(n-1)$ is the restriction of the Euclidean Hessian of $h_C$ to $u^\perp$.

We will need the following version of Lemma \ref{lem:exact} for mixed discriminants.


\begin{lemma}
\label{lem2.5}
Let $E_1,\dots,E_m\subseteq\RR^n$ be linearly independent subspaces with $d_i:=\dim E_i\geq1$ and $\sum_{i=1}^m d_i=n$.
For any symmetric matrices $A_1,\dots,A_m\in\Sym(n)$ such that $A_i$ vanishes on $E_i^\perp$, $i=1,\dots,m$, one has
\begin{equation}
    \label{eq:MDprojection}
\binom{n}{d_1,\ldots,d_m}
\D_n(A_1[d_1],\ldots,A_ m[d_m])
=
|\det T|^2\prod_{i=1}^ m\det(A_i|_{E_i}),
\end{equation}
where $T:\RR^n\to\bigoplus_{i=1}^ m E_i$ is given by $Tx=(P_{E_i}x)_{i=1}^ m$.
\end{lemma}
\begin{proof}
The operator $T$ block diagonalizes
$\sum_{i=1}^m t_i A_i$ for any $t_1,\dots,t_m\in\RR$. Hence
$$
\det\left(\sum_{i=1}^ m t_iA_i\right)
=
\det\left(T^t\operatorname{diag}(t_iA_i|_{E_i})_{i=1}^ m T\right)
=
|\det T|^2\prod_{i=1}^ m t_i^{d_i}\det(A_i|_{E_i}),
$$
where $\operatorname{diag}(t_iA_i|_{E_i})_{i=1}^ m$ is the block diagonal matrix where the $i$-th block is $t_iA_i|_{E_i}$.
The conclusion  now follows by comparing coefficients.
\end{proof}

 In Section \ref{sec4.1}, we will see that \eqref{eq:MDprojection} gives rise, via \eqref{eq:MDrep}, to a local projection formula for mixed area measures.

\subsubsection{Monotonicity of mixed volumes and support of mixed area measures}
Using the representation of mixed volumes by mixed area measures \eqref{eq:area-integral}, we immediately obtain the monotonicity inequality for mixed volumes: 
for $K,L\in\calK^n$ with $K\subseteq L$, we have $h_K\leq h_L$ and hence
$$
\V_n(K,C_1,\dots,C_{n-1}) \leq\V_n(L,C_1,\dots,C_{n-1}),
$$
and equality holds if and only if
$$
h_K=h_L\quad \text{on} \quad \supp \S_{C_1,\dots, C_{n-1}}.
$$
This will later be generalized in Lemma \ref{lem2.11}.

Therefore, the study of  equality cases in the monotonicity inequality is the same as the study of support of mixed area measures. 
Schneider \cite{Sch85} made an insightful conjecture which connects the support of mixed area measures to touching cones. 
Fix $C_1,\dots, C_{n-1}\in \calK^n$. We say $u\in \SS^{n-1}$ is a \emph{$(C_1,\dots, C_{n-1})$-extreme direction} if for every $I\subseteq\{1,\dots,n-1\}$, we have 
$$
\dim \left(\sum_{i\in I} T(C_i, u)^\perp\right)\geq |I|.
$$
Schneider conjectured that
$$
\supp S_{C_1,\dots, C_{n-1}} = \operatorname{cl} \left\{u\in \SS^{n-1}: u \text{ is $(C_1,\dots, C_{n-1})$-extreme}\right\}.
$$
When $C_1=\cdots =C_{n-1} = C$, $S_C:=S_{C,\dots,C}$ is the surface area measure of $C$ and the $(C,\dots,C)$-extreme directions are precisely the 0-extreme directions of $C$. It is known that Schneider's conjecture indeed holds for this case, see \cite[Theorem 4.5.3]{Sch14}.
However, in general this conjecture is still wide open.

For our purposes, we will  need the following partial result from \cite[Theorem 1.21]{vHW25}: if $K,C_1,\ldots,C_{n-2}\in\calK^n$, $u\in\SS^{n-1}$, $\dim T(K,u)=1$, and $v\in T(K,u)^\perp\setminus\{0\}$, then
$$
 u\in\supp \S^{v^\perp}_{\proj_{v^\perp}C_1,\ldots,\proj_{v^\perp}C_{n-2}}\quad\Longrightarrow\quad
 u\in\supp \S_{K,C_1,\ldots,C_{n-2}}.
$$
More precisely, we will make use of the following corollary of this result.

\begin{corollary}\label{cor2.9}
Let $E\subseteq\RR^n$ be an $m$-dimensional subspace, where $2\leq m\leq n$, and let $C_1,\dots,C_{m-1}\in\calK^n$.
Let $B$ be the Euclidean unit ball in $\RR^n$. 
We have
$$
\supp\S^E_{\proj_E C_1,\dots,\proj_E C_{m-1}}\subseteq\supp\S_{B[n-m],C_1,\dots,C_{m-1}}.
$$
\end{corollary}
\begin{proof}
We induct on $n-m$, the case $n=m$ being immediate.
Suppose $n>m$ and fix $u\in\supp\S^E_{\proj_E C_1,\dots,\proj_E C_{m-1}}$.
Choose a unit vector $v\in E^\perp$.
Since $\proj_{v^\perp}B$ is the unit ball in $v^\perp$ and $\proj_E\proj_{v^\perp}=\proj_E$, the induction hypothesis gives
$$
u\in\supp\S^{v^\perp}_{(\proj_{v^\perp}B)[n-m-1],\proj_{v^\perp}C_1,\dots,\proj_{v^\perp}C_{m-1}}.
$$
Since $T(B,u)=\RR_{\geq0}u$ and $v\perp u$, \cite[Theorem~ 1.21]{vHW25} implies $u\in\supp\S_{B[n-m],C_1,\dots,C_{m-1}}$.
\end{proof}

\subsubsection{Equality in monotonicity inequality} We recall the following well-known reformulation of the equality cases of the Alexandrov--Fenchel inequality, see \cite[Theorem 7.4.2]{Sch14} or {\cite[Lemma~3.12(a)]{SvH20}}. 
\begin{lemma}
\label{lem3.1}
    Let $K,L,C_1,\dots, C_{n-2}\in\mathcal{K}^n$. Assume $\V_n(L,L,C_1,\dots,C_{n-2})>0$ and
    $$
    \V_n(K,L,C_1,\dots, C_{n-2})^2 = \V_n(K,K,C_1,\dots, C_{n-2})\V_n(L,L,C_1,\dots, C_{n-2}).
    $$
    Then $\S_{K,C_1,\dots, C_{n-2}} = \alpha \S_{L,C_1,\dots, C_{n-2}} $
    holds with  $\alpha =\frac{\V_n(K,L,C_1,\dots,C_{n-2})}{\V_n(L,L,C_1,\dots,C_{n-2})}.$
\end{lemma}

By iterating this lemma, we obtain a characterization of the monotonicity inequality in terms of the support of mixed area measures in the following lemma, which is  \cite[Theorem~7.4.6]{Sch14}. We include a more direct proof here for self-containedness.
\begin{lemma}\label{lem2.11}
Fix an integer $1\leq k\leq n$.
Let $K,L,C_1,\dots, C_{n-k}\in\mathcal{K}^n$ be convex bodies such that $K\subseteq L$ and $\V_n(L[k],C_1,\dots,C_{n-k})>0$.
Denote by $B$ the unit Euclidean ball in $\RR^n$.
Then
$$
\V_n(K[ k],C_1,\ldots,C_{n- k})=\V_n(L[ k],C_1,\ldots,C_{n- k})
$$
    holds  if and only if
    $$
    h_K=h_L\quad\text{on}\quad\operatorname{supp}\S_{B[ k-1],C_1,\ldots,C_{n- k}}.
    $$
\end{lemma}

\begin{proof}
We first show the ``only if'' part by induction on $ k$.
    The base case is $ k=1$. By \eqref{eq:area-integral}, 
    $$
    \V_n(K,C_1,\dots,C_{n-1})-\V_n(L,C_1,\dots,C_{n-1})=
    \frac{1}{n}\int_{\mathbb S^{n-1}} (h_K-h_L) d\S_{C_1,\dots, C_{n-1}}.
    $$
    Since $K\subseteq L$, we have $h_K\leq h_L$. Therefore, the difference on the left-hand side vanishes if and only if $h_K=h_L$ on $\supp \S_{C_1,\dots, C_{n-1}}$.

    Now assume the conclusion holds for $ k-1$ and write $\mathcal C$ for the tuple $(C_1,\dots, C_{n-k}) $.
    By assumption and monotonicity of mixed volumes, we have that for all $1\leq i\leq  k-1$, 
    $  \V_n(K[i],L[ k-i],\mathcal C)  $
    is equal to the same positive constant.
    In particular,
    \begin{align*}        
    \V_n &(K,L,K[i-1],L[ k-i-1],\mathcal C)^2\\
    &=\V_n(K,K,K[i-1],L[ k-i-1],\mathcal C)\V_n(L,L,K[i-1],L[ k-i-1],\mathcal C).
    \end{align*}
    By Lemma \ref{lem3.1},  this implies $\S_{K,K[i-1],L[ k-i-1],\mathcal C} = \S_{L,K[i-1],L[ k-i-1],\mathcal C}$, and so by \eqref{eq:area-integral} we have 
    $$
    \V_n(K[i],L[ k-i-1], \mathcal{C}, B)=\V_n(K[i-1],L[ k-i], \mathcal{C}, B).
    $$
    Since this holds for all $1\leq i\leq  k-1$, we obtain $
    \V_n(K[ k-1], \mathcal{C}, B)=\V_n(L[ k-1], \mathcal{C}, B)$, and hence, by induction hypothesis, 
    $$
    h_K = h_L \quad \text{on} \quad \supp \S_{B[ k-2],\mathcal{C}, B} = \supp \S_{B[ k-1],\mathcal{C}}.
    $$
    
    For the ``if'' part, recall the following classical fact from \cite[Lemma 7.6.15]{Sch14}:
    $$
    \supp S_{C_1,C_2,\dots, C_{n-1}}\subseteq \supp S_{B, C_2,\dots, C_{n-1}}.
    $$
    This together with the assumption implies
    $$
    \V_n(K[ k],\mathcal C)=\V_n(K[ k-1],L,\mathcal C).
    $$
    Iterating the same argument then yields $\V_n(K[ k],\mathcal C)=\V_n(L[ k],\mathcal C)$.
\end{proof}

\section{Proof of higher-order Godbersen inequality}
\label{sec3}

We start by proving the higher-order Godbersen inequality \eqref{eq:higher-godbersen}. The method is essentially the same as the  monotonicity argument used in \cite{KM26} to prove the case $p=1$.

\begin{proposition}
\label{pro:higherGodbersen}
Let $p\geq1$ and $0\leq k_0,\dots, k_p\leq n$ be integers such that $\sum_{i=0}^pk_i=n$. For any $K\in\mathcal{K}^n$ we have
\begin{equation}
\label{eq:higherG}
 \V_{pn}\bigl(-\Delta_p K[n-k_0],\iota_1K[n-k_1],\ldots,
                    \iota_pK[n-k_p]\bigr)
 \leq\frac{(n!)^p}{(pn)!}\binom{n}{k_0,\ldots,k_p}\V_n(K)^p.
\end{equation}    
\end{proposition}

\begin{proof}
If $\V_n(K)=0$, then both sides of \eqref{eq:higherG} are equal to zero. We may thus assume  $\V_n(K)>0$. For brevity, we will denote $N:=(p+1)n$. Let $\widetilde\iota_0,\dots,\widetilde\iota_p:\RR^n\to\RR^N$ be the canonical inclusions. Write $\RR^N=\RR^n\times\RR^{pn}$  and denote by $\pi:\RR^N\to\RR^{pn}$ the projection to the second factor. Consider further the linear map $\Phi:\RR^N\to \RR^N$ given by
$$
\Phi(x_0,\dots,x_p):=(x_0,x_1+x_0,\dots,x_p+x_0).
$$
Then we have $\det\Phi=1$ and $\Phi^{-1}(x_0,\dots,x_p)=(x_0,x_1-x_0,\dots,x_p-x_0)$. Moreover,
$$
\Phi\circ\widetilde\iota_0=\Delta_{p+1},\qquad \pi\circ\Phi^{-1}\circ\widetilde\iota_0=-\Delta_p,\qquad\text{and}\qquad \pi\circ\Phi^{-1}\circ\widetilde\iota_l=\iota_l,\quad l=1,\dots,p.
$$
Consequently, \eqref{eq:mv_covariance} and \eqref{eq:MVprojection} yield
\begin{align}
\label{eq:higherGproof1}
\begin{split}
    &\V_N\big(\Delta_{p+1}K[n],\widetilde\iota_0K[n-k_0],\dots,\widetilde\iota_pK[n-k_p]\big)\\
    &\quad=\V_N\big(\Phi(\widetilde\iota_0K)[n],\widetilde\iota_0K[n-k_0],\dots,\widetilde\iota_pK[n-k_p]\big)\\
    &\quad=\V_N\left(\widetilde\iota_0K[n],\Phi^{-1}(\widetilde\iota_0K)[n-k_0],\dots,\Phi^{-1}(\widetilde\iota_pK)[n-k_p]\right)\\
    &\quad={N\choose n}^{-1}\V_n(K)\,\V_{pn}\big(-\Delta_p K[n-k_0],\iota_1 K[n-k_1],\dots,\iota_p K[n-k_p]\big).
    \end{split}
\end{align}
On the other hand, using the monotonicity of mixed volumes, $K^{p+1}=\sum_{l=0}^p\widetilde\iota_lK$, and Lemma~\ref{lem:exact}, we have
\begin{align}
\label{eq:higherGproof2}
\begin{split}
    &\V_N\big(\Delta_{p+1}K[n],\widetilde\iota_0K[n-k_0],\dots,\widetilde\iota_pK[n-k_p]\big)\\
    &\quad\leq\V_N\big(K^{p+1}[n],\widetilde\iota_0K[n-k_0],\dots,\widetilde\iota_pK[n-k_p]\big)\\
    &\quad={n\choose k_0,\dots,k_p}\V_N\big(\widetilde\iota_0K[n],\dots,\widetilde\iota_pK[n]\big)\\
    &\quad={n\choose k_0,\dots,k_p}\frac{(n!)^{p+1}}{N!}\V_n(K)^{p+1}.
    \end{split}
\end{align}
Since $\V_n(K)>0$, \eqref{eq:higherG} follows at once by comparing \eqref{eq:higherGproof1} and \eqref{eq:higherGproof2}.
\end{proof}

\section{Reduction to the face lattice condition}
\label{sec4}

In this section we prove Proposition \ref{prop1.6} which subsumes Proposition \ref{prop1.3} as a special case. We first prove the local projection formula and then use the support of mixed area measures to obtain the face lattice condition in Proposition \ref{prop1.6}.
Since mixed volumes are translation invariant and $K$ is assumed to be full-dimensional in Theorem \ref{thm1.1}, we always assume $K\in \calK^n_o$.

\subsection{The local projection formula for mixed area measures}\label{sec4.1}
In order to prove a local version of the projection formula \eqref{eq:exact}, we will need two simple lemmas.

\begin{lemma}
\label{lem:detP}
For nonzero $u, v\in\RR^n$, we have
$$
\left|\det\left(P_{u^\perp}|_{v^\perp}\right)\right|
=
\frac{|\langle u,v\rangle|}{\|u\|\|v\|}.
$$
\end{lemma}
\begin{proof}
We may assume $u$ and $v$ are not parallel, otherwise the claim is trivial. 
Decompose $\RR^n = \operatorname{span}(u,v) \oplus (u^\perp\cap v^\perp)$.
The projection  $P_{u^\perp}|_{v^\perp}$ is the identity on $u^\perp\cap v^\perp$ while for $x\in v^\perp\cap \operatorname{span}(u,v)$ we have $\|P_{u^\perp}x\| =\frac{|\langle u,v\rangle|}{\|u\|\|v\|}\|x\|$. Therefore, $\left|\det\left(P_{u^\perp}|_{v^\perp}\right)\right|
= \frac{|\langle u,v\rangle|}{\|u\|\|v\|}$.
\end{proof}

\begin{lemma}
\label{lem4.1}
Let $E\subseteq \RR^n$ be a linear subspace, $C\subseteq E$ a convex body whose support function is $C^2$  on $E\setminus\{0\}$, and $u\in \mathbb S^{n-1}$ such that $P_E u\neq 0$. Consider the decomposition $u^\perp=(E\cap u^\perp)\oplus W$, where $W$ is the orthogonal complement of  $E\cap u^\perp$ inside $u^\perp$. Then
$$
D^2 h_C(u) = \frac{1}{\|P_E u\|}D^2_E h_C\left(\frac{P_E u}{\|P_E u\|}\right)\oplus \mathbf{0}_W,
$$
where $D^2_E h_C$ is first computed inside $E$ and then naturally embedded into $\RR^n$.
    
\end{lemma} 

\begin{proof}
    Since $C\subseteq E$, it holds that $h_C(x) = h_C(P_E x)$. Therefore, by the chain rule and the $(-1)$-homogeneity of $\nabla^2_E h_C$,
    \begin{equation}
        \label{eq20}
    \nabla^2 h_C(u) =  P_E^t \nabla^2_E h_C(P_E u) P_E =\frac{1}{\|P_E u\|}P_E^t \nabla^2_E h_C\left(\frac{P_E u}{\|P_E u\|}\right)P_E
    \end{equation}
    where $\nabla^2_E h_C$ is the Euclidean Hessian of $h_C$ computed inside $E$.
    Now restricting \eqref{eq20} to $u^\perp$, we obtain 
    $$
    D^2h_C(u) = \frac{1}{\|P_E u\|}P_{E\cap u^\perp}^t D^2_E h_C\left(\frac{P_E u}{\|P_E u\|}\right)P_{E\cap u^\perp},
    $$
    which is exactly the desired formula.
\end{proof}

We now proceed to the local projection formula for mixed area measures, see \eqref{eq25} below. Although the formula  looks cumbersome,
the core of the proof is simply the projection formula \eqref{eq:MDprojection} for mixed discriminants and the main technical difficulty lies in computing the Jacobian. Similar formulas have appeared in  \cite[Theorem~10.1]{SvH20} for $p=1$ and  \cite[Corollary~4.17]{HMU26} for  mixed Monge--Ampère measures.

Fix a direction $w\in \mathbb S^{n-1}$ and subspaces $E_0,\dots, E_p\subseteq\RR^n$ such that $\RR^n=E_0+\cdots+E_p$, $w\in E_i$ for all $i$, and $V_i:=E_i\cap w^\perp$ are linearly independent. 
    Define 
    $$
    J: w^\perp\rightarrow  V_0\oplus \cdots \oplus  V_p,\qquad Jx := (P_{V_i}x)_{i=0}^p.
    $$
    Let $\mathbb S^{n-1}_+:=\{x\in \SS^{n-1}: \langle x,w\rangle > 0\}$ and $\mathbb S^{E_i}_+ := \SS^{n-1}_+\cap E_i$, and consider the diffeomorphism 
    $$
    \pi : \mathbb S^{n-1}_+\rightarrow 
     \mathbb S^{E_0}_+ \times \cdots \times \mathbb S^{E_p}_+ , \qquad \pi (u): = (u_0,\dots,u_p),\quad\text{where}\quad u_i:=\frac{P_{E_i}u}{\|P_{E_i}u\|}.
    $$
    Finally, write $d_i := \dim E_i$ and observe that  $\sum_{i=0}^p d_i=n+p$.
    
\begin{proposition}\label{prop:local_proj}
   For any convex bodies $C_i\subseteq E_i$, $i=0,\dots,p$, and any $u\in \mathbb S^{n-1}_+$ we have  
    \begin{align} \label{eq25}
    \begin{split}
    \binom{n-1}{d_0-1,\dots, d_p-1}  & \langle u , w\rangle d\S_{C_0[d_0-1],\dots, C_p[d_p-1]}(u)\\
    = &|\det J| \pi^* \left(\langle u_0 , w\rangle d\S^{E_0}_{C_0}(u_0)\otimes \cdots \otimes \langle u_p , w\rangle d\S^{E_p}_{C_p}(u_p)\right).
   \end{split}
    \end{align}

\end{proposition}

\begin{proof}
We may assume that each $C_i$ is full-dimensional in $E_i$ and $h_{C_i}$ is $C^2$  in $E_i\setminus\{0\}$. The general case follows by approximation, using the weak continuity of mixed area measures with respect to Hausdorff convergence. For the proof we will denote $c_{d_0,\dots,d_p}:=\binom{n-1}{d_0-1,\ldots,d_p-1}$.

Let us first consider the case where the subspaces $V_i$ are mutually orthogonal. Then $|\det J|=1$. By \eqref{eq:MDrep},
for $u\in\mathbb S^{n-1}_+$ we have
$$
\frac{d\S_{C_0[d_0-1],\ldots,C_p[d_p-1]}}{d\mathcal H^{n-1}}(u)
=
\D_{n-1}\bigl(
D^2h_{C_0}(u)[d_0-1],\ldots,D^2h_{C_p}(u)[d_p-1]
\bigr).
$$
Applying Lemmas \ref{lem2.5} and \ref{lem4.1} to the right-hand side, we obtain
$$
c_{d_0,\dots,d_p}
\D_{n-1}\bigl(
D^2h_{C_0}(u)[d_0-1],\ldots,D^2h_{C_p}(u)[d_p-1]
\bigr)=
|\det T|^2
\prod_{i=0}^p
\frac{\det\left(
D^2_{E_i}h_{C_i}(u_i)
\right)}{\|P_{E_i}u\|^{d_i-1}},
$$
where $T:u^\perp\to\bigoplus_{i=0}^p(E_i\cap u^\perp) $ is defined by $Tx=(P_{E_i\cap u^\perp}x)_{i=0}^p$.
We claim that
\begin{align}
\label{eq:detT}
|\det T|=\frac{\langle u,w\rangle^p}{\prod_{i=0}^p\|P_{E_i}u\|}.
\end{align}
Indeed, the adjoint $T^t$ of $T$ is the addition map $T^t(x_0,\dots, x_p)=x_0+\cdots+ x_p$ and the linear map $P_{w^\perp}\circ T^t:\bigoplus_{i=0}^p(E_i\cap u^\perp)\to\bigoplus_{i=0}^p V_i$ is block diagonal. Since $w\in E_i$, the $i$-th block is the orthogonal projection $P_{w^\perp}:(P_{E_i}u)^\perp\to w^\perp$, where the orthogonal complements are taken inside $E_i$, and hence has determinant $|\det P_{w^\perp}|=\frac{\langle P_{E_i}u,  w\rangle}{\|P_{E_i}u\|}=\frac{\langle u,  w\rangle}{\|P_{E_i}u\|}$ by Lemma \ref{lem:detP}. Therefore, 
 $$
\left|\det(P_{w^\perp}\circ T^t)\right|=\prod_{i=0}^p\frac{\langle u,w\rangle}{\|P_{E_i}u\|}
=\frac{\langle u,w\rangle^{p+1}}{\prod_{i=0}^p\|P_{E_i}u\|}.
$$
At the same time, Lemma \ref{lem:detP} also directly implies
$$
\left|\det(P_{w^\perp}\circ T^t)\right|=\left|\det (P_{w^\perp}|_{u^\perp})\right||\det (T)| = \langle u,w\rangle|\det T|.
$$
Combining the two equations shows \eqref{eq:detT}. Consequently,
\begin{align}
\label{eq:LPF1}
c_{d_0,\dots,d_p} d\S_{C_0[d_0-1],\ldots,C_p[d_p-1]}(u)=\langle u,w\rangle^{2p}\prod_{i=0}^p\frac{\det\left(
D^2_{E_i}h_{C_i}(u_i)
\right)}{\|P_{E_i}u\|^{d_i+1}}d\mathcal H^{n-1}(u).
\end{align}
It remains to compute the Jacobian of $\pi$. To this end, for any $x\in u^\perp$ we have
$$
D\pi(u)x=\left(\frac{P_{E_i\cap u^\perp}x}{\|P_{E_i}u\|}\right)_{i=0}^p.
$$
That is, $D\pi(u)$ is the product of $T$ and a block diagonal matrix and thus \eqref{eq:detT} implies
$$
|\det D\pi(u)|
=|\det T|\prod_{i=0}^p\|P_{E_i}u\|^{-(d_i-1)}
=\frac{\langle u,w\rangle^p}{\prod_{i=0}^p\|P_{E_i}u\|^{d_i}}.
$$
Since $\langle u_i,w\rangle=\frac{\langle u,w\rangle}{\|P_{E_i}u\|}$, the change of variables formula and \eqref{eq:MDrep} inside each $E_i$ yield
\begin{align}
\label{eq:LPF2}
&\pi^*\left(\bigotimes_{i=0}^p\langle u_i,w\rangle\,d\S^{E_i}_{C_i}(u_i)\right)=
\langle u,w\rangle^{2p+1}\prod_{i=0}^p\frac{\det\left(D^2_{E_i}h_{C_i}(u_i)\right)}{\|P_{E_i}u\|^{d_i+1}}
\,d\mathcal H^{n-1}(u).
\end{align}
Comparing \eqref{eq:LPF1} with \eqref{eq:LPF2} proves the proposition in the orthogonal case.

If the subspaces $V_i=E_i\cap w^\perp$ are not necessarily orthogonal, choose an invertible linear map $A:\RR^n\to\RR^n$ fixing $w$, whose restriction to each $E_i$ is an isometry, and such that the subspaces $A(E_i\cap w^\perp)$ are mutually orthogonal.
Since $A|_{w^\perp}\circ J^t$ is an isometry, we have $|\det A||\det J|=1$. For 
$u\in\SS^{n-1}_+$ set $\phi(u):=\frac{A^{-t}u}{\|A^{-t}u\|}$. 
By \eqref{eq:area-covariance} and $\langle\phi(u),w\rangle=\frac{\langle u,w\rangle}{\|A^{-t}u\|}$, we obtain
\begin{equation}
\label{eq22}
\langle\phi(u),w\rangle\,d\left(\phi^*\S_{AC_0[d_0-1],\dots,AC_p[d_p-1]}\right)(u)
=|\det A|\langle u,w\rangle\,d\S_{C_0[d_0-1],\dots,C_p[d_p-1]}(u).
\end{equation}
Let $\pi_A(u):=\left(\frac{P_{AE_iu}}{\|P_{AE_iu}\|}\right)_{i=0}^p$ denote the projection map for the subspaces $AE_i$. Then, by the orthogonal case, the left-hand side of \eqref{eq22} is equal to
$$
c_{d_0,\dots,d_p}^{-1}\phi^*\pi_A^*\left(\bigotimes_{i=0}^p\langle v_i,w\rangle\,d\S^{AE_i}_{AC_i}(v_i)\right)=c_{d_0,\dots,d_p}^{-1}\pi^*\left(\bigotimes_{i=0}^p\langle u_i,w\rangle\,d\S^{E_i}_{C_i}(u_i)\right),
$$
where we used $\pi_A\circ\phi=(A,\dots,A)\circ\pi$, which follows from $P_{AE_i}A^{-t}u=AP_{E_i}u$, and \eqref{eq:area-covariance} again for the isometries $A|_{E_i}$. Now \eqref{eq25} follows at once using $|\det A|^{-1}=|\det J|$.
\end{proof}

\subsection{Proof of Propositions \ref{prop1.3} and \ref{prop1.6}}
We directly prove Proposition \ref{prop1.6}, which includes Proposition \ref{prop1.3} as the special  case $p=1$.
\begin{proof}[Proof of Proposition \ref{prop1.6}]
We will keep the notation from the proof of Proposition \ref{pro:higherGodbersen}. In particular, $N=(p+1)n$, and $\widetilde\iota_0,\dots,\widetilde\iota_p:\RR^n\to\RR^N$ are the canonical inclusions. By \eqref{eq:higherGproof1} and \eqref{eq:higherGproof2}, a full-dimensional convex body $K\in\calK^n$ attains equality in \eqref{eq:higher-godbersen} if and only if
$$
\begin{aligned}
&\V_N\bigl(\Delta_{p+1}K[n],\widetilde\iota_0K[n-k_0],\dots,\widetilde\iota_pK[n-k_p]\bigr)\\
&\qquad=\V_N\bigl(K^{p+1}[n],\widetilde\iota_0K[n-k_0],\dots,\widetilde\iota_pK[n-k_p]\bigr)>0.
\end{aligned}
$$
By Lemma \ref{lem2.11}, this is equivalent to
\begin{equation}
h_{\Delta_{p+1}K}=h_{K^{p+1}}\quad\text{on}\quad\supp\S_{B[n-1],\widetilde\iota_0K[n-k_0],\dots,\widetilde\iota_pK[n-k_p]},
\label{eq10}
\end{equation}
where $B$ is the unit ball in $\RR^N$.

Assume first that \eqref{eq10} holds and fix nonzero $u_0,\dots,u_p\in\RR^n$ such that $\dim T(K,u_i)\leq k_i$. Without loss of generality, we may normalize so that $u:=(u_0,\dots,u_p)\in\mathbb S^{N-1}$.
For each $i$, choose a subspace $E_i\subseteq T(K,u_i)^\perp$ of dimension $n-k_i$ and denote
$$
E:=\left(\bigoplus_{i=0}^p\widetilde\iota_iE_i\right)\oplus\RR u,
\qquad E_i':=\widetilde\iota_iE_i\oplus\RR u.
$$
Then $\dim E=pn+1$, $\dim E_i'=n-k_i+1$, and the subspaces $E_i'\cap u^\perp=\widetilde\iota_iE_i$ are mutually orthogonal.
Moreover, $P_E\widetilde\iota_iK=P_{E_i'}\widetilde\iota_iK$ is full-dimensional in $E_i'$.
By \eqref{eq7} and the choice of $E_i$, we have
$$
T_{E_i'}(P_{E_i'}\widetilde\iota_iK,u)
=T(\widetilde\iota_iK,u)\cap E_i'
=\RR_{\geq0}u.
$$
Indeed, $\widetilde\iota_i x+tu\in T(\widetilde\iota_iK,u)$ with $x\in E_i$ if and only if $x+tu_i\in T(K,u_i)$, which forces $x=0$ and $t\geq0$.
Thus $u$ is a $0$-extreme direction of $P_{E_i'}\widetilde\iota_iK$ for every $i$.
Since the support of the surface area measure is the closure of the $0$-extreme directions \cite[Theorem~4.5.3]{Sch14}, Proposition \ref{prop:local_proj} gives
$$
u\in\supp\S^E_{P_E\widetilde\iota_0K[n-k_0],\dots,P_E\widetilde\iota_pK[n-k_p]}.
$$
Now $\dim E^\perp=n-1$.
By Corollary \ref{cor2.9}, we obtain
$$
u\in\supp\S_{B[n-1],\widetilde\iota_0K[n-k_0],\dots,\widetilde\iota_pK[n-k_p]}.
$$
By \eqref{eq10}, $h_{\Delta_{p+1}K}(u)=h_{K^{p+1}}(u)$, that is,
$$
\max_{x\in K}\sum_{i=0}^p\langle x,u_i\rangle
=\sum_{i=0}^p\max_{x\in K}\langle x,u_i\rangle.
$$
Any maximizer on the left-hand side thus belongs to $\bigcap_{i=0}^p \{x\in K: \langle x, u_i\rangle = h_K(u_i)\}$.
Choose such a point $x$ and a unit vector $w\in\relint N(K,x)$.
Then $N(K,x)=T(K,w)$ contains every $u_i$, and hence every $T(K,u_i)$, which completes the proof for the ``only if'' direction.

For the ``if'' direction, assume that, conversely, the condition on touching cones holds.
By \cite[Corollary 1.9]{vHW25}, in order  to prove \eqref{eq10}, it suffices to show $h_{\Delta_{p+1}K}(u)=h_{K^{p+1}}(u)$ for every $(B[n-1],\widetilde\iota_0K[n-k_0],\dots,\widetilde\iota_pK[n-k_p])$-extreme direction $u\in\mathbb S^{N-1}$. Let $u=(u_0,\dots,u_p)$ be such a direction. Then, by definition, $\dim T(\widetilde\iota_iK,u)^\perp\geq n-k_i$ for all $i$. Since
$$
T(\widetilde\iota_iK,u)^\perp=\widetilde\iota_i\left(T(K,u_i)^\perp\right),
$$
we have $\dim T(K,u_i)\leq k_i$.
Applying the assumption on touching cones, we obtain $w\in\RR^n$ such that $T(K,w)$ contains every $T(K,u_i)$.
Fix $x\in K$ with $\langle x,w\rangle=h_K(w)$.
Then for every $i$, $u_i\in T(K,w)\subseteq N(K,x)$ and, in particular, $h_K(u_i)=\langle x,u_i\rangle$.
Consequently,
$$
h_{\Delta_{p+1}K}(u)\geq\sum_{i=0}^p\langle x,u_i\rangle=\sum_{i=0}^p h_K(u_i)=h_{K^{p+1}}(u).
$$
Since the reverse inequality holds for any $u\in\SS^{N-1}$ because $\Delta_{p+1}K\subseteq K^{p+1}$, this finishes the proof.
\end{proof}

\section{The two face characterization of simplices}
\label{sec5}

In this section we prove Theorem~\ref{thm:equivalent} by inductively reducing the incidence condition (b) to the case $k=1$.
This will complete the proof of Theorem~\ref{thm1.1} through Proposition~\ref{prop1.3}.

Throughout the section, we  write $\aff(A,B):=\aff (A\cup B)$ and $\conv(A,B):=\conv (A\cup B)$ for the affine and convex hull, respectively, of nonempty sets
$A,B\subseteq\RR^n$. Moreover, when $B=\{x\}$ is a singleton, we write simply $\conv(A,x)$ instead of $\conv(A,\{x\})$. Since we will repeatedly work with the condition (b) of Theorem~\ref{thm:equivalent}, it will be convenient to capture it by a definition.

\begin{definition}\label{def:Ipq}
Let $K\in\calK^n$, and let $p,q\geq0$ be integers.
We say that $K$ satisfies property $(I_{p,q})$ if for any two faces $F$ and $G$ of $K$ with $\dim F\leq p$ and $\dim G\leq q$, the face $F\vee G$ is proper.
\end{definition}


\subsection{Induction on $k$}
We will show that any $K\in\calK^n$ with property $(I_{k-1,n-k-1})$, where $2\leq k\leq n-1$, also satisfies $(I_{k-2,n-k})$.
More precisely, we assume that for any faces $F$ and $G$ with $\dim F\leq k-1$ and $\dim G\leq n-k-1$, $F\vee G$ is a proper face.
We will show that under this assumption, for any faces $F$ and $G$ with $\dim F\leq k-2$ and $\dim G \leq n-k$, $F\vee G$ is a proper face.
Note that we are imposing a stronger dimension condition  on $F$ and a weaker dimension condition on $G$.

The key observation for the induction step is that any proper face $G'$ of $G$ has dimension at most $\dim G-1 \leq n-k-1$, which allows one to apply $(I_{k-1,n-k-1})$ to $F$ and $G'$.
This gives a proper face $F\vee G'$ containing $F$ and meeting $G$.
The difficulty here is that, \textit{a priori}, one has no control on the dimension of $F\vee G'$.
Therefore, $\dim (F\vee G')$ can take any value between $\dim F$ and $n-1$.
The following lemma deals with the case where $\dim (F\vee G')\leq k-1$ by taking $T = F\vee G'$.

\begin{lemma}\label{lem:T}
Suppose that $K\in\calK^n$ satisfies $(I_{k-1,n-k-1})$, where $1\leq k\leq n-1$.
Let $F$ and $G$ be  faces of $K$ with $\dim G=n-k$.
If there is a face $T$ of $K$ such that $F\subseteq T$, $\dim T\leq k-1$, and $T\cap G\neq\emptyset$, then $F\vee G$ is proper.
\end{lemma}
\begin{proof}
Assume such $T$ exists. If $T\cap\relint G\neq \emptyset$, then $G\subseteq T$ by Lemma~\ref{lemma2.1}, and the conclusion follows.
Otherwise, choose $y\in T\cap \relbd G$ and $g\in\relint G$. Then there exists  $z\in\relbd G$ such that $g\in(y,z)$. 
Since $z\in\relbd G$, $F(K,z)$ is a proper face of $G$, and hence $\dim F(K,z)\leq n-k-1$.
Property $(I_{k-1,n-k-1})$ then implies that $T\vee F(K,z)$ is proper. Since $y,z\in T\vee F(K,z)$, we have $g\in \relint G\cap\big( T\vee F(K,z)\big)$, and Lemma~\ref{lemma2.1} gives $G\subseteq T\vee F(K,z)$.
As $F\subseteq T\vee F(K,z)$ as well, $F\vee G$ is a face of $T\vee F(K,z)$, which is proper.
Therefore $F\vee G$ is a proper face of $K$.
\end{proof}

We are therefore left with the case where $\dim(F\vee G')\geq k$.

\begin{lemma}
\label{lem:2cases}
Let $K\in\calK^n$ be a full-dimensional convex body. Let $F$ and $G$ be proper faces of $K$ satisfying the following properties: 
\begin{align}
\label{eq:subface_property_F}
       \text{if } F'\text{ is a proper face of }F,\text{ then }F'\vee G\text{ is a proper face of }K;\\
\label{eq:subface_property_G}
       \text{if } G'\text{ is a proper face of }G,\text{ then }F\vee G'\text{ is a proper face of }K.
    \end{align}
Then  one of the following holds:
\begin{enuma}
\item $F\vee G $ is proper;
\item $F$ and $G$ are affinely independent, $K\cap \aff(F,G)=\conv(F,G)$, and for any extreme point $y\in G$ one has
\begin{align}
\label{eq:convFy}
(F\vee y)\cap\aff(F,G)=\conv(F,y).    
\end{align}
\end{enuma}
\end{lemma}
\begin{proof}

We fix $u\in\relint F$ and $v\in\relint G$ for the rest of the proof and put $o:=(u+v)/2$.
Then we have $o\in \relint \conv(F,G)$. If $o\in\partial K$, we can find a proper face $H$ of $K$ containing $o$. Then $u,v\in H$, and Lemma~\ref{lemma2.1} implies that $H$ must also contain $F$ and $G$. Therefore $F\vee G\subseteq  H$ is a proper face of $K$, and item (a) holds in this case. 
In the following we will assume $o\in \operatorname{int} K$ and aim to prove the second claim. We will denote $S:=\aff(F, G)$ and $Q:=\conv(F, G)$.

To prove that $F$ and $G$ are affinely independent, we must show $\dim S=\dim F+\dim G+1$. To this end, we consider the convex body
$$
\widehat Q:=\left\{((1-t)f,tg,t):f\in F,\ g\in G,\ 0\leq t\leq1\right\}\in\calK^{2n+1},
$$
i.e., the convex hull of $F\times\{0\}\times\{0\}$ and $\{0\}\times G\times\{1\}$, and the linear map $\pi:\RR^n\times\RR^n\times\RR\to\RR^n$ given by $\pi(a,b,t):=a+b$. Observe that $\pi(\widehat Q)=Q$, hence $\pi(\aff\widehat Q)=S$. Because $\dim \aff\widehat Q=\dim F+\dim G+1$, it suffices to show that $\pi|_{\aff\widehat Q}$ is injective.

We will first show that $\pi(\relbd\widehat Q)\subseteq\partial K$. We have
$$
\relbd \widehat Q=R_1\cup R_2\cup \left(\bigcup_{f\in\relbd F}R_{3,f}\right)\cup\left(\bigcup_{g\in\relbd G}R_{4,g}\right),
$$
where 
\begin{align*}
    R_1&:=F\times\{0\}\times\{0\},\\
    R_2&:=\{0\}\times G\times\{1\},\\
    R_{3,f}&:=\left\{\big((1-t)f,tg,t\big):0<t<1, \,g\in G\right\},\quad f\in\relbd F,\\
    R_{4,g}&:=\left\{\big((1-t)f,tg,t\big):0<t<1,\, f\in F\right\}, \quad g\in\relbd G.
\end{align*}
Since  $F$ and $G$ are proper faces of $K$, $\pi(R_1)$ and $\pi(R_2)$ are contained in $\partial K$. Since every $f\in\relbd F$ lies in a proper face $F'$ of $F$, \eqref{eq:subface_property_F} implies $\pi(R_{3,f})\subseteq F'\vee G\subseteq\partial K$. Similarly, using \eqref{eq:subface_property_G}, we get $\pi(R_{4,g})\subseteq\partial K$ for every $g\in\relbd G$. 
We thus conclude $\pi(\relbd\widehat Q)\subseteq\partial K$.

Consider the point $\widehat o:=(u/2,v/2,1/2)\in\relint\widehat Q$. It satisfies  $\pi(\widehat o)=o$.
Fix an arbitrary nonzero  $w\in \RR^n\times \RR^n\times \RR$ such that $\widehat o+ w\in \aff\widehat Q$.
Then the ray  $\{\widehat o+tw$: $t\in (0,+ \infty)\}$ intersects the relative boundary of $\widehat Q$.
That is,  there is $r>0$ such that $\widehat o+rw\in\relbd\widehat Q$.
Using $\pi(\relbd\widehat Q)\subseteq\partial K$, we obtain $\pi(\widehat o+rw) \in \partial K$.
But $\pi(\widehat o) = o  \in \operatorname{int} K$.
Therefore, $\pi(w) = r^{-1}(\pi(\widehat o+rw) - \pi(\widehat o))\neq 0$.
Since $w$ was chosen arbitrarily,  $\pi$ is injective on $\aff\widehat Q$.
We have thus proven that $F$ and $G$ are affinely independent.

For the  claim $K\cap S = Q$, note first that $K\cap S\supseteq Q$ holds because $S = \aff Q$.
For the reverse inclusion, assume for contradiction that there exists $x\in (K\cap S)\setminus Q$.
In particular,   $x\in \aff Q\setminus Q $.
Since $o\in \relint Q$, the open segment $(o,x)$ intersects $\relbd Q\subseteq\partial K$ at some point $y$.
However, since $o\in \operatorname{int} K$ and $x\in K$, we have $(o,x)\subseteq \operatorname{int} K$, a contradiction.
We thus conclude  $K\cap S = Q$.

Finally, choose an extreme point $y\in G$ and set $D:=F\vee y$. To prove \eqref{eq:convFy}, observe first that both sides are faces of $\conv (F,G)$. Indeed, on the one hand, $\conv(F,y)\subseteq\conv (F,G)$ is a face since $y$ is an extreme point of $G$.
On the other hand, $D$ is a face of $K$, so $D\cap S$ is a face of $K\cap S=\conv(F,G)$. Next, put $p:=(u+y)/2$.
Since $y\notin\aff F$, we have $p\in\relint \conv(F,y)$. By Lemma~\ref{lemma2.2}, $p$ also lies in $\relint D$ and hence in $\relint(D\cap S)$. Consequently, $\relint (D\cap S) \cap \relint \conv(F,y)\neq \emptyset$. Now Lemma \ref{lemma2.1} applied twice implies $D\cap S= \conv(F,y)$, as claimed.
\end{proof}

\begin{lemma}\label{lemma4.5}
Let $2\leq k\leq n-1$, and suppose that $K\in\calK^n$ satisfies $(I_{k-1,n-k-1})$.
Let $F$ and $G$ be faces of $K$ with $\dim F=k$ and $\dim G=n-k$, such that $F\cap G=\{y\}$ for some $y\in K$ and $\aff(F, G)=\RR^n$.
Then $\conv(F,G)=K$.
\end{lemma}
\begin{proof}
Translate $K$ so that $y=0$. Then $U:=\aff F$ and $V:=\aff G$ are complementary linear subspaces, that is, $\RR^n=U\oplus V$. In particular, $\dim\conv(F,G)=n$. Since $\conv(F,G)\subseteq K$ it suffices to show $\partial (\conv(F,G))\subseteq\partial K$.

Observe that we in fact have $0\in\relbd F\cap\relbd G$.
Indeed, if for instance $0\in \relint F$, Lemma \ref{lemma2.1} implies that $F\subseteq G$, and  hence $U$ and $V$ cannot be linearly independent.

We claim that every point $z\in\partial (\conv(F,G))$ lies in a segment $[f,g]$ with $f\in\relbd F$ and $g\in\relbd G$.
The case $z=0$ is immediate, since $0\in\relbd F\cap\relbd G$.
If $z\in F\setminus\{0\}$, the ray from $0$ through $z$ intersects $\relbd F$ at $f$.
Then $z\in[0,f]$ as required.
The case $z\in G\setminus\{0\}$  follows similarly.
Otherwise, write $z=tf_0+(1-t)g_0$ with $0<t<1$, $f_0\in F$, and $g_0\in G$.
If $f_0\in \relbd F$ and $g_0\in \relbd G$, we are already done.
Hence we assume for example $f_0\in \relint F$.
Let $H$ be a proper face of $\conv(F,G)$ containing $z$.
Then $f_0,g_0\in H$, and in particular $f_0 \in H\cap \relint F$.
By  Lemma~\ref{lemma2.1},  $F\subseteq H$.
So $0\in H$ and $H\cap G$ is a proper face of $G$.
Extend the ray from $0$ through $f_0$ to $f\in\relbd F$, and write $f_0=\rho f$ with $0<\rho<1$.
Set $g:=\frac{1-t}{1-t\rho}g_0$.
Then $g\in[0,g_0]\subseteq H\cap G\subseteq\relbd G$ and $z=t\rho f+(1-t\rho)g$. This shows the existence of $f\in \relbd F,\, g\in \relbd G $ with $z\in [f,g]$ in this case. The case $g_0\in\relint G$ is analogous.

Fix $z\in \partial (\conv(F,G))$ and choose $f,g$ as above.
Since $f\in \relbd F$, we have $\dim F(K,f)\leq\dim F-1 = k-1$.
Similarly, $\dim F(K,g)\leq n-k-1$.
Property $(I_{k-1,n-k-1})$ therefore implies that $f,g$ are contained in a proper face of $K$.
In particular, $z\in [f,g]\subseteq F(K, \{f,g\})\subseteq\partial K$.
Since $z$ was chosen arbitrarily, we conclude $\partial (\conv(F,G))\subseteq\partial K$.

\end{proof}

\begin{lemma}\label{lemma4.6}
Let $1\leq k\leq n-1$, and let $D,G\in \calK^n$ with $\dim D=k$ and $\dim G=n-k$, such that $0 \in D\cap G$ and $D\cup G$ spans $\RR^n$.
Let $F$ be a face of $D$ with $0 \notin F$.
Then there is a face $T$ of $\conv(D,G)$ containing $F$, meeting $G$, and satisfying $\dim T=\dim F+1$.
\end{lemma}
\begin{proof}
Put $C:=\conv(D,G)$, $U:=\aff D$, and $V:=\aff G$.
Since $0\in D\cap G$, the hypotheses give $\RR^n=U\oplus V$.
Choose a nonzero extreme point $w$ of $G$, and set $T:=\conv(F,w)$.
Our choice of $w$ implies $w\in T\cap G$ and $w\notin U$, so $\dim T=\dim F+1$.

To verify the face property, suppose $z\in T$ and $z=\lambda x_1+(1-\lambda)x_2$, where $x_1,x_2\in C$, $0<\lambda<1$. We need to show that $x_1,x_2\in T$.

Write $x_i=\alpha_i d_i+(1-\alpha_i)g_i$, where $d_i\in D$, $g_i\in G$, and $0\leq\alpha_i\leq1$, for $i=1,2$. Put $r:=\lambda\alpha_1+(1-\lambda)\alpha_2$. 
Since $z\in T$, we also have $z=t f+(1-t)w$, for some $f\in F$ and $0\leq t\leq1$.
Together, we have
$$
t f+(1-t)w =z = \lambda (\alpha_1 d_1+(1-\alpha_1)g_1)+(1-\lambda)(\alpha_2 d_2+(1-\alpha_2)g_2).
$$
Taking the components in $U$ and $V$ respectively gives
\begin{align}
\label{eq11}
t f
&=\lambda\alpha_1d_1+(1-\lambda)\alpha_2d_2,\\
(1-t)w
&=\lambda(1-\alpha_1)g_1+(1-\lambda)(1-\alpha_2)g_2.
\label{eq12}
\end{align}
We claim that this implies $r=t$. Indeed, if $r<t$, then \eqref{eq11} implies
$$
f
=\frac{\lambda\alpha_1}{t}d_1
+\frac{(1-\lambda)\alpha_2}{t}d_2
+\left(1-\frac rt\right)0.
$$
Hence $f$ is a convex combination of points where $0$ has a positive coefficient.
Since $f\in F$, we have $0\in F$, a contradiction. If $r>t$, then \eqref{eq12} implies
\[
w=\frac{\lambda(1-\alpha_1)}{1-t}g_1
+\frac{(1-\lambda)(1-\alpha_2)}{1-t}g_2
+\frac{r-t}{1-t}0.
\]
However, $w$ is an extreme point, hence cannot be written as a nontrivial convex combination of points in $G$, a contradiction. We thus have $r=t$ and will distinguish three cases.

First, if $t=0$, then $r=0$ forces $\alpha_1=\alpha_2=0$. Thus $x_1,x_2\in G$, and \eqref{eq12} together with extremality of $w$ in $G$ gives $x_1=x_2=w\in T$.

Second, if $t=1$, then $r=1$ forces $\alpha_1=\alpha_2=1$. Thus $x_1,x_2\in D$, and \eqref{eq11} together with the face property of $F$ in $D$ gives $x_1,x_2\in F\subseteq T$.

Finally, assume $0<t<1$. Then by \eqref{eq11} and \eqref{eq12} we have the convex combinations
\begin{align}
\label{eq11convex}
f &=\frac{\lambda\alpha_1}{t}d_1 +\frac{(1-\lambda)\alpha_2}{t}d_2,\\
w&=\frac{\lambda(1-\alpha_1)}{1-t}g_1 +\frac{(1-\lambda)(1-\alpha_2)}{1-t}g_2.
\label{eq12convex}
\end{align}
If $\alpha_i>0$, then \eqref{eq11convex} together with the face property of $F$ in $D$ implies $d_i\in F$. If $\alpha_i<1$, then \eqref{eq12convex} together with extremality of $w$ in $G$ implies $g_i=w$. An easy case distinction then shows that $x_i=\alpha_i d_i+(1-\alpha_i)g_i\in\conv(F,w)$, $i=1,2$, for any $(\alpha_1,\alpha_2)\in[0,1]^2$.
\end{proof}

\begin{lemma}
\label{lem:with_subface}
    Let $2\leq k\leq n-1$ and $0\leq d\leq k-2$. Let $K\in\calK^n$ be a full-dimensional convex body satisfying $(I_{k-1,n-k-1})$. Assume that $F$ and $G$ are faces of $K$ satisfying $\dim F=d$, $\dim G=n-k$, and \eqref{eq:subface_property_F}. Then $F\vee G$ is proper.
\end{lemma}

\begin{proof}
    Let $G'$ be a proper face of $G$. Then $\dim G'\leq n-k-1$ and $(I_{k-1,n-k-1})$ implies that $F\vee G'$ is proper. Hence, $F$ and $G$ also satisfy \eqref{eq:subface_property_G} and we may use Lemma \ref{lem:2cases}. If item (a) in Lemma~\ref{lem:2cases} holds, then we are done. 
Therefore, we assume from now on that item (b) holds. Choose an extreme point $y$ of $G$ and denote $D:=F\vee y$. We will also denote $S:=\aff(F,G)$. Then $F$ and $G$ are affinely independent, $K\cap S = \conv (F, G)$, and $D\cap S = \conv(F,y)$.

Property $(I_{k-1,n-k-1})$ implies that $D$ is proper. Moreover, since $F$ and $G$ are affinely independent, we have
\begin{align*}
    \dim S &= \dim F+\dim G+1 = d+n-k+1,\\
    \dim \conv(F,y) &= \dim F+1 = d+1,
\end{align*}
and therefore
\begin{equation}\label{eq13}
\dim D\leq \dim (D\cap S)+\codim(S) = 
\dim\conv(F,y)+n-\dim S=k.
\end{equation}
If $\dim D\leq k-1$, then taking $T=D$ in Lemma~\ref{lem:T} shows that $F\vee G$ is proper as desired.

It remains to treat the case $\dim D=k$. In this case, \eqref{eq13} takes equality, which yields $\aff (D,S)=\RR^n$. Since $F\subseteq D$ and $S=\aff(F,G)$, we have $\aff(D, G)=\RR^n$.
Moreover, $D\cap G=\conv(F,y)\cap G=\{y\}$.
We may therefore apply Lemma~\ref{lemma4.5} to $D,G$, which  gives $K=\conv(D,G)$. Since $F$ is a proper face of $D$ and translating $K$ so that $y=0$, Lemma~\ref{lemma4.6} implies that there is a face $T$ of $K$ containing $F$, intersecting $G$, and satisfying $\dim T=d+1\leq k-1$. Applying Lemma~\ref{lem:T} to this $T$ shows that $F\vee G$ is proper.
\end{proof}

Next we will show by induction that the assumption \eqref{eq:subface_property_F} on the faces $F$ and $G$ in the previous lemma can in fact be omitted.

\begin{lemma}
\label{lem:without_subface}
    Let $2\leq k\leq n-1$ and $0\leq d\leq k-2$. Let $K\in\calK^n$ be a full-dimensional convex body satisfying $(I_{k-1,n-k-1})$. Assume that $F$ and $G$ are faces of $K$ satisfying $\dim F=d$ and $\dim G=n-k$. Then $F\vee G$ is proper.
\end{lemma}

\begin{proof}
    We argue by strong induction on $d$. First, if $d=0$, then there is no proper face of $F$. Hence \eqref{eq:subface_property_F} holds trivially and the claim follows immediately from Lemma \ref{lem:with_subface}. Second, fix $d>0$ and assume the statement holds for $0,\dots,d-1$. For every proper face $F'$ of $F$, we have $\dim F'<d$ and hence the induction hypothesis implies that $F'\vee G$ is proper. Thus \eqref{eq:subface_property_F} holds and the claim follows again from  Lemma \ref{lem:with_subface}.
\end{proof}

Now we can finally proceed to the proof of the induction step in Theorem \ref{thm:equivalent}.

\begin{proposition}\label{prop4.8}
Let $K\in\calK^n$ be a full-dimensional convex body and $2\leq k\leq n-1$. If $K$ satisfies $(I_{k-1,n-k-1})$, then it satisfies $(I_{k-2,n-k})$.
\end{proposition}
\begin{proof}
Take any faces $F,G$ of $K$ with $\dim F\leq k-2$ and $\dim G\leq n-k$. If $\dim G\leq n-k-1$, property $(I_{k-1,n-k-1})$ gives that $F\vee G$ is proper. Otherwise, $\dim G=n-k$, and the same conclusion follows from Lemma~\ref{lem:without_subface}, applied with $d=\dim F$. Since $F$ and $G$ were arbitrary, $K$ satisfies $(I_{k-2,n-k})$.
\end{proof}

\subsection{The $k=1$ case}

It remains to prove the base case. Recall that we assume $n\geq2$.

\begin{lemma}\label{lem:induction_base}
Let $K\in\calK^n$ be a full-dimensional convex body. If $K$ satisfies $(I_{0,n-2})$, then $K$ is a simplex.
\end{lemma}

\begin{proof}
Let $x$ be an arbitrary extreme point of $K$ and put $F:=\{x\}$. Choose $p\in\interior K$ and let $b\in\partial K$ be the (unique) point such that $p\in[b,x]$. Denote $G:=F(K,b)$.
Then $F\vee G$ contains $p\in \interior K$ and thus $F\vee G=K$.
Since $K$ satisfies $(I_{0,n-2})$,  we have  $\dim G=n-1$. Moreover, for any proper face $G'$ of $G$, one has $\dim G'\leq n-2$ and hence by $(I_{0,n-2})$, $F\vee G'$ is proper. This shows that \eqref{eq:subface_property_G} holds. Note also that \eqref{eq:subface_property_F} holds trivially since $F$ has no proper faces.
We may therefore apply Lemma~\ref{lem:2cases}. 
Since  $F\vee G=K$, item (a) of Lemma~\ref{lem:2cases} cannot hold. 
Hence $F$ and $G$ are affinely independent and one has
$$
K = K\cap \aff(F,G) =\conv(F,G) = \conv (G,x).
$$
We thus conclude that $K$ is a pyramid with base $G$ and apex $x$.
In particular, every extreme point of $K$ other than $x$ belongs to $G$ and so $x$ does not lie in the affine hull of the other extreme points.

Since $x$ was arbitrary, we get that every finite set of extreme points of $K$ is affinely independent and consequently $K$ has at most $n+1$ extreme points. Since $K$ is full-dimensional and equal to the convex hull of its extreme points, we conclude that $K$ is an $n$-simplex.
\end{proof}

\subsection{Proof of Theorems \ref{thm:equivalent} and \ref{thm1.1}}
\begin{proof}[Proof of Theorem~\ref{thm:equivalent}]
First, assume that $K$ satisfies (b) or, in other words, the property $(I_{k-1,n-k-1})$. Applying Proposition~\ref{prop4.8} repeatedly, this implies that $K$ satisfies $(I_{0,n-2})$. 
Now Lemma~\ref{lem:induction_base} shows that $K$ is a simplex.

Conversely, assume that $K$ is a simplex and $F,G$ are faces of $K$ with $\dim F\leq k-1$ and $\dim G\leq n-k-1$. Since $F$ and $G$ are simplices themselves, they together contain at most $k+(n-k)=n$ vertices of $K$. Hence there is a vertex $y\in K$ not contained in $F\cup G$. Choose $H$ to be the facet of $K$ not containing $y$. Then $F\vee G\subseteq H$ and (b) follows. 

Next, translate $K$ so that it contains the origin in its interior and assume that $K^{\circ}$ satisfies (b). Consider $u, v\in\mathbb S^{n-1}$ such that $\dim T(K, u)\leq k$ and $\dim T(K, v)\leq n-k$. By Lemma \ref{lem:fact}, $F:=T(K, u)\cap\partial K^{\circ}$ and $G:=T(K, v)\cap\partial K^{\circ}$ are faces of $K^\circ$ satisfying $\dim F\leq k-1$ and $\dim G\leq n-k-1$. Therefore, the face $F\vee G$ is proper. Using Lemma  \ref{lem:fact}, we have $\mathbb R_{\geq0}(F\vee G)=T(K, w)$ for some $w\in\mathbb S^{n-1}$. The touching cone $T(K, w)$ clearly contains both $T(K, u)$ and $T(K, v)$, and we conclude that $K$ satisfies (c).

Conversely, assume that $K$ satisfies (c), and let $F, G$ be proper faces of $K^{\circ}$ such that $\dim F\leq k-1$ and $\dim G\leq n-k-1$. By Lemma \ref{lem:fact}, there exist $u, v\in\mathbb S^{n-1}$ with $\mathbb R_{\geq0}F=T(K, u)$ and $\mathbb R_{\geq0}G=T(K, v)$. Then $\dim T(K, u)\leq k$ and $\dim T(K, v)\leq n-k$, and hence condition (c) in turn ensures that there exists some $w\in\mathbb S^{n-1}$ such that $T(K, w)$ contains both $T(K, u)$ and $T(K, v)$. By Lemma \ref{lem:fact}, $T(K, w)\cap\partial K^{\circ}$ is a proper face of $K^\circ$. Since it clearly contains both $F$ and $G$, $F\vee G$ is proper. We conclude that $K^{\circ}$ satisfies (b). 

Finally, observing that $K\in\calK^n_o$ is a simplex if and only if $K^{\circ}$ is a simplex and using the previously shown equivalence (a)$\iff$(b), we have that $K$ satisfies (c) if and only if $K$ is a simplex. This finishes the proof.
\end{proof}

\begin{proof}[Proof of Theorem~\ref{thm1.1}]
This follows immediately by combining Proposition~\ref{prop1.3} with Theorem~\ref{thm:equivalent}.
\end{proof}

\section{Equality cases of the higher-order Godbersen inequality}
\label{sec6}
In this last section we prove Theorem~\ref{thm:higher_equivalent} by reducing the incidence condition for $p+1$ faces to the two-face condition of Theorem~\ref{thm:equivalent}. Combined with Proposition~\ref{prop1.6}, this completes the proof of equality characterization in Theorem~\ref{thm:higherGodbersen}.

The following proposition, which we will require for the reduction, is most probably known. However, since we have no suitable reference, we provide a proof here for completeness.

\begin{proposition}\label{pro:reduction_higher_godbersen}
    Let $K\in\calK^n$ be a full-dimensional convex body. Let $f,g \geq 1$ be integers such that $f+g \geq n+1$. Then there exist proper faces $F,G$ of $K$ such that 
    \[\dim F \leq f-1, \qquad \dim G \leq g-1, \qquad\text{and}\qquad F \vee G = K.\]
\end{proposition}
\begin{proof}
    Without loss of generality we may assume $f+g = n+1$. 
    We will inductively construct two sequences  $(F_j)$ and $(G_j)$, $j=0,1,\dots$, of proper faces of $K$ satisfying $F_j\vee G_j=K$, $\dim F_j\leq f-1$ and \eqref{eq:subface_property_F},  \eqref{eq:subface_property_G} from Lemma \ref{lem:2cases}.
    
First, we may always write $K = x_1 \vee \dots \vee x_m $ where $m$ is an integer and $x_1,\dots,x_m$ are extreme points of $K$. Let us fix such a representation of $K$ with minimal $m$. We let $F_0 := \{x_1\}$ and $G_0$  be a minimal face of $x_2 \vee \dots \vee x_m$ such that $F_0\vee G_0=K$. It is clear that $F_0,G_0$ are faces of $K$ with the required properties.

Second, assume that $F_j,G_j\subseteq K$ are as above. If $\dim G_j\leq g-1$, we set $F_{j+1}:=F_j$ and $G_{j+1}:=G_j$. If $\dim G_j\geq g$, we proceed as follows: Pick a maximal proper face $\widetilde G\subseteq G_j$. Then there exists an extreme point $y\in G_j$ not contained in $\widetilde G$. We set $\widetilde F:=F_j\vee y$. By Lemma \ref{lem:2cases}, $F_j$ and $G_j$ are affinely independent and satisfy \eqref{eq:convFy}. The former implies
$$
\dim\aff(F_j, G_j) = \dim F_j + \dim G_j + 1,
$$
the latter yields $\aff(\widetilde F\cap\aff(F_j, G_j))=\aff(F_j,y)$. Consequently,
\begin{align*}
    \dim \widetilde F&\leq\dim\aff(\widetilde F\cap\aff(F_j,G_j))+\codim\aff(F_j,G_j)\\
    &=\dim F_j+1+n-(\dim F_j+\dim G_j+1)\\
    &=n-\dim G_j\\
    &\leq f-1,
    \end{align*}
where in the last step we used our assumptions $f+g=n+1$ and $\dim G_j\geq g$. Using the maximality of $\widetilde G$, we further have
$$
\widetilde F\vee \widetilde G=F_j\vee y\vee \widetilde G=F_j\vee G_j=K.
$$
We choose faces $F_{j+1}\subseteq \widetilde F$ and $G_{j+1}\subseteq \widetilde G$ satisfying $F_{j+1}\vee G_{j+1}=K$ and minimizing $\dim F_{j+1}+\dim G_{j+1}$. Then we have $\dim F_{j+1}\leq \dim\widetilde F\leq f-1$, and the minimality of $\dim F_{j+1}+\dim G_{j+1}$ implies that $F_{j+1},G_{j+1}$ satisfy \eqref{eq:subface_property_F} and \eqref{eq:subface_property_G}. Note that $F_{j+1},G_{j+1}$ are clearly proper faces of $K$ since $F_j,G_j$ were such.

 Observe that, if $\dim G_j\geq g$, we have $\dim G_{j+1}<\dim G_j$ since $\widetilde G\subseteq G_j$ was proper. Therefore, there exists $j_0\geq0$ such that $\dim G_{j_0}\leq g-1$ and we thus may take $F:=F_{j_0}$ and $G:=G_{j_0}$ to finish the proof.
\end{proof}

It will be convenient to generalize Definition \ref{def:Ipq} to the higher-order setting.

\begin{definition}\label{def:higherIp}
Let $K\in\calK^n$ and fix integers $p\geq1$ and $q_0,\dots,q_p\geq0$.
We say that $K$ satisfies property $(I_{q_0,\dots, q_p})$ if for any faces $F_0,\dots,F_p$ of $K$ with $\dim F_i\leq q_i$, $i=0,\dots,p$, the face $F_0\vee\cdots\vee F_p$ is proper.
\end{definition}

\begin{lemma}
\label{lem:p_reduction}
Let $K\in\calK^n$, $p\geq2$, and $q_0,\dots,q_p\geq0$. If $K$ satisfies $(I_{q_0,\dots, q_p})$, then it satisfies $(I_{q_0,\dots,q_{p-2},q_{p-1}+q_p+1})$.
\end{lemma}

\begin{proof}
   Assume that $K$ satisfies $(I_{q_0,\dots, q_p})$ and consider any faces $F_0,\dots,F_{p-1}\subseteq K$ with
    $$\dim F_i\leq q_i,\quad i=0,\dots,p-2,\qquad\text{and}\qquad \dim F_{p-1}\leq q_{p-1}+q_p+1.
    $$
    By Proposition \ref{pro:reduction_higher_godbersen}, there exist faces $F,G\subseteq F_{p-1}$ such that $\dim F\leq q_{p-1}$, $\dim G\leq q_p$, and $F\vee G=F_{p-1}$. By $(I_{q_0,\dots, q_p})$, $
    F_0\vee\cdots\vee F_{p-1}=F_0\vee\cdots\vee F_{p-2}\vee F\vee G
    $
    is a proper face of $K$, and hence $K$ satisfies $(I_{q_0,\dots,q_{p-2},q_{p-1}+q_p+1})$.
\end{proof}

\begin{proof}[Proof of Theorem \ref{thm:higher_equivalent}]
First, assume that $K$ satisfies (b) or, in other words, the property $(I_{k_0-1,\dots, k_p-1})$. Applying Lemma \ref{lem:p_reduction} repeatedly and using $\sum_{i=0}^pk_i=n$, this implies that $K$ satisfies $(I_{k_0-1,n-k_0-1})$. 
Now Theorem \ref{thm:equivalent} implies that $K$ is a simplex.

Conversely, assume that $K$ is a simplex and $F_0,\dots,F_p$ are faces of $K$ with $\dim F_i\leq k_i-1$, $i=0,\dots,p$. Since $F_0,\dots,F_p$ are simplices themselves, they together contain at most $\sum_{i=0}^p k_i=n$ vertices of $K$. Hence there is a vertex $y\in K$ not contained in any $F_i$, $i=0,\dots,p$. Choosing $F\subseteq K$ to be the facet not containing $y$, (b) follows. 

Finally, translate $K$ so that it contains the origin in its interior. Then a straightforward generalization of the corresponding part of the proof of Theorem~\ref{thm:equivalent}, that is, an easy application of Lemma \ref{lem:fact}, shows that $K$ satisfies (c) if and only if $K^\circ$ satisfies (b). By the previously shown equivalence (a)$\iff$(b), we thus have that $K$ satisfies (c) if and only if $K^\circ$ is a simplex. Using the self-duality property of a simplex, we conclude that (a)$\iff$(c). This finishes the proof.
\end{proof}

\begin{proof}[Proof of Theorem \ref{thm:higherGodbersen}]
    The inequality \eqref{eq:higher-godbersen} was proven in Proposition \ref{pro:higherGodbersen}. By Remark \ref{rem:reduction}, it remains to characterize the equality cases when $K$ is full-dimensional and $k_0, \dots, k_p\geq1$ (and hence $p\leq n-1$). In this case, it follows immediately from Proposition \ref{prop1.6} and Theorem \ref{thm:higher_equivalent} that $K$ attains equality in \eqref{eq:higher-godbersen} if and only if $K$ is a simplex.
\end{proof}

\bibliographystyle{amsplain}
\bibliography{bibliography}
\end{document}